\documentclass[a4paper,12pt]{amsart}

\usepackage{mathdots,url,amsmath,amsthm}
\usepackage{pdflscape}
\usepackage{cite}

\newcommand{\HH}{\mathbb{H}}
\newcommand{\M}{\mathcal{M}}
\newcommand{\QM}{\mathcal{QM}}
\newcommand{\C}{\mathbb{C}}
\newcommand{\N}{\mathbb{N}}
\newcommand{\id}{\mathrm{id}}
\newcommand{\trans}{\mathsf{T}}
\newcommand{\Kup}{K^{\mathrm{up}}}
\newcommand{\Kdown}{K^{\mathrm{down}}}
\newtheorem{theorem}{Theorem}
\newtheorem{lemma}[theorem]{Lemma}
\newtheorem{proposition}[theorem]{Proposition}

\theoremstyle{definition}
\newtheorem{definition}[theorem]{Definition}
\newtheorem{remark}[theorem]{Remark}

\numberwithin{equation}{section}
\numberwithin{theorem}{section}
\allowdisplaybreaks
\begin{document}


%
%

\title{Quasimodular forms as solutions of modular
  differential equations}

\author{Peter J. Grabner}
\thanks{The author is supported by the Austrian
    Science Fund FWF project F5503 (part of the Special Research Program (SFB)
    ``Quasi-Monte Carlo Methods: Theory and Applications'')}

\address{Institut f\"ur Analysis und Zahlentheorie,
Technische Universit\"at Graz,\\
Kopernikusgasse 24,
8010 Graz,
Austria}
\email{peter.grabner@tugraz.at}

\maketitle


\begin{abstract}
  We study quasimodular forms of depth $\leq4$ and determine under which
  conditions they occur as solutions of modular differential
  equations. Furthermore, we study which modular differential equations have
  quasimodular solutions. We use these results to investigate extremal
  quasimodular forms as introduced by M.~Kaneko and M.~Koike
  further. Especially, we prove a conjecture stated by these authors concerning
  the divisors of the denominators occurring in their Fourier expansion.
\end{abstract}



\section{Introduction}\label{sec:introduction}

The notion of ``quasimodular form'' was coined by M.~Kaneko and D.~Zagier in
\cite{Kaneko_Zagier1995:generalized_jacobi_theta}. Since then quasimodular
forms have gained increasing attention as they have intrinsic connections to
very different fields of mathematics and beyond. For two excellent
introductions to the subject we refer to
\cite{Royer2012:quasimodular_forms_introduction,
  Zagier2008:elliptic_modular_forms}.

There has been an extensive study of linear differential equations, whose
solution set is invariant under modular transformations
(see \cite{Kaneko_Nagatomo_Sakai2017:third_order_modular,
  Kaneko_Koike2004:quasimodular_solutions_differential,
  Kaneko_Koike2003:modular_forms_hypergeometric,
  Kawasetsu_Sakai2018:modular_linear_differential,
  Arike_Kaneko_Nagatomo+2016:affine_vertex_operator,
  Franc_Mason2016:hypergeometric_series_modular,
  Mason2007:vector_valued_modular}). Such differential equations can be used to
study families of modular forms and quasimodular forms. The question, under
which conditions such equations have modular or quasimodular solutions, is very
prominent in many of these papers. The present paper will shed some new light
on that and gives a unified view on the subject.

Quasimodular forms gained new interest since they
occurred prominently in the construction of certain Fourier eigenfunctions with
prescribed zeros. These were used in the proof that in dimensions $8$ and $24$
the $E_8$ and the Leech lattice achieve the best packing (see
\cite{Cohn-Kumar-Miller+2017:sphere_packing_24,
  Viazovska2017:sphere_packing_8}), as well as in the proof of universal
optimality of these lattices (see
\cite{Cohn_Kumar_Miller+2019:universal_optimality}). For a survey on the
construction of such Fourier eigenfunctions we refer to
\cite{Feigenbaum_Grabner_Hardin2020:eigenfunctions_fourier_transform}. In this
paper modular differential equations are used to encode the asymptotic
behavior of quasimodular forms in a concise and tractable way. The
differential equations are then used to derive linear recurrence relations for
the quasimodular forms of interest.

In a series of papers M.~Knopp and many coauthors have introduced and studied
vector valued modular forms
\cite{Franc_Mason2016:hypergeometric_series_modular,
  Franc_Mason2014:fourier_coefficients_vector,
  Knopp_Mason2012:logarithmic_vector_valued,
  Mason2012:fourier_coefficients_vector,
  Knopp_Mason2011:logarithmic_vector_valued,
  Marks_Mason2010:structure_module_vector, Mason2008:2_dimensional_vector,
  Mason2007:vector_valued_modular, Knopp_Mason2004:vector_valued_modular,
  Knopp_Mason2003:vector_valued_modular}.  Modular differential equations also
play a role in this context, as they are a method to capture properties of the
components of a vector valued form in a concise way (see
\cite{Franc_Mason2016:hypergeometric_series_modular,
  Mason2007:vector_valued_modular}). These give rise to representations of the
modular group. In this context the action of the map $T:z\mapsto z+1$ is always
diagonizable. We will study a similar concept for quasimodular forms, where the
action of $T$ will turn out not to be diagonizable.

The paper is organized as follows. In Section~\ref{sec:basics} we recall some
basic facts and definitions about modular forms and  quasimodular forms. We
shortly recall the Frobenius ansatz method for finding holomorphic solutions of
differential equations.

In Section~\ref{sec:quasi-modul-vect} we introduce and study quasimodular
vectors as an analogue to vector valued modular forms. We derive several
properties that have been known for the modular case and will be used later in
this paper.

In Section~\ref{sec:balanc-quas-forms} we introduce the notion of
\emph{balanced} quasimodular forms. These are forms $f$, which exhibit certain
patterns for the vanishing orders at $i\infty$ of the quasimodular forms
occurring in the transformation behavior of $f$ under the modular group. We
find estimates for the vanishing orders of such forms and prove that they are
solutions of modular differential equations if their depth is $\leq4$. There is
only one degenerate exception to this rule, namely powers of $\Delta$, which
occur as solutions of modular differential equations of any order.

In Section~\ref{sec:modul-diff-equat} we study modular differential equations
which have (balanced) quasimodular solutions. We give a full description of all
differential equations of order $\leq5$ with quasimodular solutions.

In Section~\ref{sec:extr-quas-forms} we use the results of
Section~\ref{sec:modul-diff-equat} to derive differential recursions for
extremal quasimodular forms. Such forms are defined by the property that they
have maximal possible order of vanishing at $i\infty$ (see
\cite{Kaneko_Koike2006:extremal_quasimodular_forms}). The recursions obtained
are then used to prove a conjecture stated in
\cite{Kaneko_Koike2006:extremal_quasimodular_forms} concerning the divisors of
the denominators of the Fourier coefficients of such forms of depth
$\leq4$. Recently, a different proof of this conjecture for extremal
quasimodular forms of depth $1$ and weight divisible by $6$ was given by
F.~Pellarin and G.~Nebe
\cite{Pellarin_Nebe2019:extremal_quasi_modular}. A.~Mono
\cite{Mono2020:conjecture_kaneko_koike} gave an independent proof for depth
$1$, which could also cover weights in the residue classes $\equiv2,4\pmod6$.

In an Appendix we collect some huge expressions for polynomials and modular
forms that occur in Section~\ref{sec:extr-quas-forms} for the case of depth
$4$.
\section{Basics}\label{sec:basics}
In this section we collect some basic facts about modular and quasimodular
forms and give a short exposition of the Frobenius ansatz method to solve
linear differential equations.
\subsection{Modular forms}\label{sec:modular-forms}
The modular group $\Gamma$ is the group of $2\times2$-matrices with integer
entries and determinant $1$
\begin{equation*}
  \Gamma=\mathrm{PSL}(2,\mathbb{Z})=\left\{
    \begin{pmatrix}
      a&b\\c&d
    \end{pmatrix}\Bigm| a,b,c,d\in\mathbb{Z}, ac-bd=1
  \right\}/\{\pm I\}.
\end{equation*}
The group $\Gamma$ is generated by
\begin{equation}
  \label{eq:ST}
  Sz=-\frac1z\quad Tz=z+1,
\end{equation}
which satisfy the relations $S^2=\id$ and $(ST)^3=\id$.  It acts on the upper
half plane $\mathbb{H}=\{z\in\mathbb{C}\mid\Im z>0\}$ by M\"obius
transformation
\begin{equation*}
  \begin{pmatrix}
      a&b\\c&d
    \end{pmatrix}z=\frac{az+b}{cz+d}.
  \end{equation*}
A holomorphic function $f:\mathbb{H}\to\mathbb{C}$ is called a \emph{weakly
  holomorphic modular form of weight $w$}, if it satisfies
\begin{equation}
  \label{eq:modular}
  (cz+d)^{-w}f\left(\frac{az+b}{cz+d}\right)=f(z)
\end{equation}
for all $z\in\mathbb{H}$ and all $\bigl( \begin{smallmatrix}a & b\\ c &
  d\end{smallmatrix}\bigr)\in\Gamma$. 

The vector space of weakly holomorphic modular
forms is denoted by $\M_w^!(\Gamma)$. This space is non-trivial only
for even values of $w$. A form $f$ is called holomorphic, if
\begin{equation*}
  f(i\infty):=\lim_{\Im z\to+\infty}f(z)
\end{equation*}
exists. The subspace $\M_w(\Gamma)$ of holomorphic modular forms is
non-trivial only for even $w\geq4$. Its dimension equals
\begin{equation*}
  \dim\M_w(\Gamma)=
  \begin{cases}
    \left\lfloor\frac w{12}\right\rfloor&\text{for }w\equiv 2\pmod{12}\\
    \left\lfloor\frac w{12}\right\rfloor+1&\text{otherwise.}
  \end{cases}
\end{equation*}
The most prominent examples of modular forms are the Eisenstein series
\begin{equation}
  \label{eq:eisenstein-2k}
  E_{2k}(z)=\frac1{2\zeta(2k)}\sum\limits_{(m,n)\in\mathbb{Z}\setminus\{(0,0)\}}
  \frac1{(mz+n)^{2k}}
\end{equation}
for $k\geq2$, which are modular forms of weight $2k$. They admit a Fourier
expansion (setting $q=e^{2\pi iz}$ as usual in this context)
\begin{equation}\label{eq:eisenstein-fourier}
  E_{2k}=1-\frac{4k}{B_{2k}}\sum_{n=1}^\infty\sigma_{2k-1}(n)q^n,
\end{equation}
where $\sigma_{2k-1}(n)=\sum_{d\mid n}d^{2k-1}$ denotes the divisor sum of
order $2k-1$ and $B_{2k}$ denote the Bernoulli numbers. The defining series
\eqref{eq:eisenstein-2k} does not converge for $k=1$ in the given
form. Nevertheless, the series \eqref{eq:eisenstein-fourier} converges for
$k\geq1$.  This entails a slightly more complicated transformation behavior
under the action of $S$
\begin{equation}
  \label{eq:E2S}
  z^{-2}E_2(Sz)=E_2(z)+\frac6{\pi iz}.
\end{equation}
Every modular form can be expressed as a polynomial in $E_4$ and $E_6$.
Furthermore, by the invariance under $T$, every holomorphic modular form $f$
has a Fourier expansion
\begin{equation*}
  f(z)=\sum_{n=0}^\infty a_f(n)e^{2\pi nz}=\sum_{n=0}^\infty a_f(n)q^n.
\end{equation*}
In the sequel we will follow the convention to freely switch between dependence
on $z$ and $q$.

A holomorphic form $f$ is called a \emph{cusp form}, if $f(i\infty)=0$. The
prototypical example of a cusp form is
\begin{equation}
  \label{eq:Delta}
  \Delta=\frac1{1728}\left(E_4^3-E_6^2\right).
\end{equation}
The space of cusp forms is denoted by $\mathcal{S}_w(\Gamma)$. Since we only
deal with modular forms for the full modular group $\Gamma$, we will omit
reference to the group in the sequel.

For a detailed introduction to the theory of modular forms we refer to
\cite{Shimura2012:modular_forms_basics,
  Berndt_Knopp2008:heckes_theory_modular,
  Bruinier_Geer_Harder+2008:1_2_3-modular,
  Stein2007:modular_forms_computational,
  Diamond_Shurman2005:first_course_modular,
  Iwaniec1997:topics_classical_automorphic,
  Lang1995:introduction_to_modular}.

\subsection{Quasimodular forms}\label{sec:quasimodular-forms}
The vector space of quasimodular forms of weight $w$ and depth $\leq r$ is
given by
\begin{equation}
  \label{eq:quasi-space}
  \QM_w^{r}=\bigoplus_{\ell=0}^rE_2^\ell\M_{w-2\ell}.
\end{equation}
Quasimodular forms occur naturally as derivatives of modular forms (see
\cite{Royer2012:quasimodular_forms_introduction,
  Zagier2008:elliptic_modular_forms,Choie_Lee2019:jacobi_like_forms}).

The dimension of the space $\QM_w^r$ will play an important role in this paper,
so we give a general formula.  In the following the notation $[P]$ means $1$,
if the condition $P$ is satisfied and $0$ otherwise.

\begin{proposition}\label{prop:dim}
  Let $w\equiv0\pmod2$ and $r\geq0$, then the dimension of the space of
  quasimodular forms equals
  \begin{equation}\label{eq:dim-qm-gen}
    \begin{split}
      \dim\QM_w^r&=\left\lfloor\frac{w(r+1)}{12}\right\rfloor-
      \left\lfloor\frac{r+1}6\right\rfloor
      \left(r-3\left\lfloor\frac{r+1}6\right\rfloor-1\right)+
      \left\lfloor\frac{r}6\right\rfloor\\
      &+1-\left[w(r+1)\equiv2\!\!\!\!\pmod{12}\right].
    \end{split}
\end{equation}
In the special case $r\leq4$ and $w(r+1)\equiv0\pmod{12}$ this simplifies to
\begin{equation}
  \label{eq:dim-qm}
  \dim\QM_w^r=\frac{w(r+1)}{12}+1.
\end{equation}
\end{proposition}
\begin{proof}
  We start with
  \begin{equation*}
  \dim\QM_w^r=\sum_{k=0}^r\dim\M_{w-2k}
\end{equation*}
and split the range of summation in intervals of length $6$ to obtain
\begin{equation}\label{eq:dim-qm-prelim}
  \dim\QM_w^r
  =\frac w2\left\lfloor\frac{r}6\right\rfloor+\dim\QM_w^{r\!\!\!\!\!\pmod6}-
  \left\lfloor\frac{r}6\right\rfloor\left(r-3\left\lfloor\frac{r}6\right\rfloor
    -2\right).
\end{equation}
The values of
the dimension for $r=0,\ldots,5$ are given by
\begin{align*}
  \dim\QM_w^0&=\left\lfloor\frac w{12}\right\rfloor+1-
  \left[w\equiv2\!\!\!\!\pmod{12}\right]\\
  \dim\QM_w^1&=\left\lfloor\frac w{6}\right\rfloor+1\\
  \dim\QM_w^2&=\left\lfloor\frac w{4}\right\rfloor+1\\
  \dim\QM_w^3&=\left\lfloor\frac w{3}\right\rfloor+1\\
  \dim\QM_w^4&=\left\lfloor\frac {5w}{12}\right\rfloor+1-
  \left[w\equiv10\!\!\!\!\pmod{12}\right]\\
  \dim\QM_w^5&=\frac w2.
\end{align*}

 From \eqref{eq:dim-qm-prelim} and the dimension formulas for $r=0,\ldots,5$
 the general formula \eqref{eq:dim-qm-gen} can be obtained by a case
 distinction $r\pmod6$ and $w\pmod{12}$.
\end{proof}

We will follow the convention to denote the derivative by
\begin{equation*}
  f'=\frac1{2\pi i}\frac{df}{dz}=q\frac{df}{dq}.
\end{equation*}
With this notation Ramanujan's identities read
\begin{equation}
  \label{eq:ramanujan}
  \begin{split}
    E_2'&=\frac1{12}\left(E_2^2-E_4\right)\\
    E_4'&=\frac13\left(E_2E_4-E_6\right)\\
    E_6'&=\frac12\left(E_2E_6-E_4^2\right).
  \end{split}
\end{equation}
These give rise to the definition of the Serre derivative
\begin{equation*}
  \partial_wf=f'-\frac w{12}E_2f,
\end{equation*}
where $w$ is (related to) the weight of $f$. We will use the product rule
\begin{equation*}
  \partial_{w_1+w_2}(fg)=\left(\partial_{w_1}f\right)g+
  f\left(\partial_{w_2}g\right)
\end{equation*}
and also make frequent use of the following immediate consequences of
\eqref{eq:ramanujan}
\begin{equation}
  \label{eq:serre-ramanujan}
  \begin{split}
    \partial_1E_2&=-\frac1{12}E_4\\
    \partial_4E_4&=-\frac13E_6\\
    \partial_6E_6&=-\frac12E_4^2\\
    \partial_{12}\Delta&=0.
  \end{split}
\end{equation}
From the second and third equation together with the fact that every
holomorphic form is a polynomial in $E_4$ and $E_6$, it follows immediately that
for a form $f\in\M_w$ we have $\partial_wf\in\M_{w+2}$, and for
$f\in\mathcal{S}_w$ we have $\partial_wf\in\mathcal{S}_{w+2}$.

The following lemma is a special case of
\cite[Proposition~3.3]{Kaneko_Koike2006:extremal_quasimodular_forms}, where a
similar result for a certain family of operators $\theta_k^{(r)}$ of order $r$
is proved.
\begin{lemma}\label{lem:quasi-serre}
  The Serre derivative $\partial_{w-r}$ maps quasimodular forms of weight $w$
  and depth $\leq r$ to quasimodular forms of weight $w+2$ and depth $\leq r$.
\end{lemma}
\begin{proof}
  A quasimodular form of weight $w$ and depth $\leq r$ can be written as
  \begin{equation}
    \label{eq:quasi-r}
    f_w=\sum_{k=0}^rA_{w-2k}E_2^k
  \end{equation}
  with $A_{w-2k}\in\M_{w-2k}$.
  Then
  \begin{equation*}
    \partial_{w-r}f_w=\sum_{k=0}^r\left(\left(\partial_{w-2k}A_{w-2k}\right)E_2^k+
    A_{w-2k}\left(\partial_{2k-r}E_2^k\right)\right).
\end{equation*}
Inserting
\begin{equation*}
  \partial_{2k-r}E_2^k=-\frac k{12}E_4E_2^{k-1}+\frac{r-k}{12}E_2^{k+1}
\end{equation*}
yields
\begin{align*}
  \partial_{w-r}f_w&=\partial_wA_w-\frac1{12}E_4A_{w-2}\\
  &+  \sum_{k=1}^{r-1}\left(\partial_{w-2k} A_{w-2k}-\frac{k+1}{12}E_4A_{w-2k-2}
    +\frac{r-k+1}{12}A_{w-2k+2}\right)E_2^k\\
  &+  \left(\partial_{w-2r}A_{w-2r}+\frac1{12}A_{w-2r+2}\right)E_2^r,
\end{align*}
which is a quasimodular form of weight $w+2$ and depth $\leq r$.
\end{proof}

\begin{lemma}\label{lem:serre}
  Let $f:\HH\to\C$ be holomorphic. Then
  \begin{equation*}
    \partial_w\left(z^{-w}f(Sz)\right)=z^{-w-2}\left(\partial_wf\right)(Sz).
  \end{equation*}
\end{lemma}
\begin{proof}
  We compute
  \begin{align*}
    &\partial_w\left(z^{-w}f(Sz)\right)
    =z^{-w-2}f'(Sz)-\frac{w}{2\pi i}z^{-w-1}f(Sz)
    -\frac{w}{12}E_2(z)z^{-w}f(Sz)\\
    = &z^{-w-2}f'(Sz)-\frac{w}{2\pi i}z^{-w-1}f(Sz)
    -\frac{w}{12}\left(z^{-2}E_2(Sz)-
      \frac{12}{2\pi iz}\right)z^{-w}f(Sz)\\
    = &z^{-w-2}\left(f'(Sz)-\frac{w}{12}E_2(Sz)f(Sz)\right)=
    z^{-w-2}\left(\partial_wf\right)(Sz).
  \end{align*}
\end{proof}

We will use the following convention for iterated Serre derivatives throughout
the paper:
\begin{equation*}
  \partial_w^0f=f,\quad
  \partial_w^{k+1}=\partial_{w+2k}\left(\partial_w^kf\right).
\end{equation*}

We will consider  differential equations of the form
\begin{equation}
  \label{eq:diff-r}
    K_{\mathbf{B}}f=B_{m}\partial_{w-r}^{r+1}f+
    B_{m+2}\partial_{w-r}^{r}f+\cdots+B_{m+2r+2}f=0,
\end{equation}
where $\mathbf{B}=(B_{m},\ldots,B_{m+2r+2})$ are modular forms of respective
weights $m,m+2,\ldots,m+2r+2$ with $B_m(i\infty)=1$.
\begin{lemma}\label{lem:invar}
  For every holomorphic solution $f:\HH\to\C$ of the differential equation
  \eqref{eq:diff-r}, $f(Tz)$ and $z^{r-w}f(Sz)$ are also solutions. Thus, for
  any $\gamma=\begin{pmatrix}a&b\\c&d\end{pmatrix}\in\Gamma$,
  $(cz+d)^{r-w}f(\gamma z)$ is also a solution.
\end{lemma}
\begin{proof}
  Let $f:\HH\to\C$ be a holomorphic solution of \eqref{eq:diff-r}. Then by
  Lemma~\ref{lem:serre} we have
  \begin{equation*}
    K_{\mathbf{B}}(z^{r-w}f(Sz))=z^{-r-w-m-2}K_{\mathbf{B}}(f)(Sz)=0.
  \end{equation*}
  Similarly, since all coefficient functions and all Serre derivatives are
  invariant under $T$, we have
  \begin{equation*}
    K_{\mathbf{B}}(f(Tz))=K_{\mathbf{B}}(f)(Tz)=0.
  \end{equation*}
  The last assertion follows from the fact that $S$ and $T$ generate $\Gamma$.
\end{proof}
We will be mostly interested in the case that $m=0$ and thus $B_m=1$, in which
we call the corresponding equation \emph{normalized}.

\subsection{The Frobenius ansatz method}\label{sec:frob-ansatz-meth}
F.~G.~Frobenius \cite{Frobenius1873:uber_integration_linearen} devised a method
to find holomorphic (=power series) solutions of differential equations of the
form
\begin{equation}\label{eq:frobenius}
  f^{(n)}(z)+a_{n-1}(z)f^{(n-1)}(z)+\cdots+a_0(z)f(z)=0,
\end{equation}
where $a_0,\ldots,a_{n-1}$ are meromorphic functions on some
$U\subset\C$. Under the condition that $a_k$ has a pole of order at most $n-k$
at $z_0\in U$ (this is called a \emph{regular singularity} in this context),
there exists a solution of \eqref{eq:frobenius} of the form
\begin{equation*}
  f(z)=(z-z_0)^\lambda\sum_{n=0}^\infty f_n(z-z_0)^n,
\end{equation*}
where $\lambda$ is a solution of the so called indicial equation, a polynomial
equation arising from inserting this ansatz into \eqref{eq:frobenius} and
requiring $f_0\neq0$. Originally, the method was developed for equations of
degree $n=2$.

In our case the situation is slightly different, since we are interested in
power series in $q=e^{2\pi iz}$, the derivatives still being with respect to
$z$. Furthermore, we have expressed all derivatives in terms of Serre
derivatives. We are looking for solutions of \eqref{eq:diff-r} of the form
\begin{equation*}
  q^\lambda\sum_{n=0}^\infty a(n)q^n.
\end{equation*}
For such a solution to exist $\lambda$ has to be a root of the \emph{indicial
  equation}
\begin{equation}
  \label{eq:indicial}
  p_{\mathbf{B}}(x)=\sum_{\ell=0}^{r+1}B_{m+2\ell}(i\infty)q_{r+1-\ell}(x,w)=0,
\end{equation}
where $q_0(x,w)=1$ and
\begin{equation*}
  q_\ell(x,w)=\left(x-\frac{w-r}{12}\right)\left(x-\frac{w-r+2}{12}\right)
  \cdots\left(x-\frac{w-r+2\ell-2}{12}\right).
\end{equation*}
Then $p_{\mathbf{B}}(x)$ is a polynomial of degree $r+1$ with roots
$\lambda_0,\ldots,\lambda_r\in\C$ called the \emph{Frobenius exponents} of
\eqref{eq:diff-r}. As long as these exponents are pairwise different and none
of the pairwise differences $\lambda_k-\lambda_\ell$ ($k\neq\ell$) is an
integer, the ansatz method immediately gives $r+1$ linearly independent
solutions of \eqref{eq:diff-r} by successively solving for
$a(1),a(2),\ldots$. In our case we are especially interested in the opposite
situation, namely that all exponents are positive integers and thus all the
pairwise differences are integers. In the classical situation this is the case,
where the monodromy representation associated to a fundamental system of
solutions is  not diagonisable in general.

We order the exponents in decreasing order
$\lambda_0\geq\lambda_1\geq\cdots\geq\lambda_r\geq0$. Then we start with the
solution
\begin{equation*}
  f_0(z)=q^{\lambda_0}\sum_{n=0}^\infty a_0(n)q^n.
\end{equation*}
A second solution can then be found using the ansatz
\begin{equation*}
  f_1(z)=Czf_0(z)+q^{\lambda_1}\sum_{n=0}^\infty a_1(n)q^n,
\end{equation*}
where $C$ has to be chosen so that the computation of the coefficient
$a_1(\lambda_0-\lambda_1)$ is possible. Further solutions can be found by
making an ansatz
\begin{equation*}
  f_\ell(z)=C_\ell^{(\ell)} z^\ell f_0(z)+
  C_{\ell-1}^{(\ell)}z^{\ell-1}q^{\lambda_1}\sum_{n=0}^\infty a_1(n)q^n+
  \cdots+q^{\lambda_\ell}\sum_{n=0}^\infty a_\ell(n)q^n,
\end{equation*}
where the constants $C_\ell^{(\ell)},\ldots,C_1^{(\ell)}$ have to be chosen so
that the computation of the coefficients
$a_\ell(\lambda_{\ell-1}-\lambda_\ell),\ldots,a_\ell(\lambda_0-\lambda_\ell)$
is possible. For more details on the method we refer to
\cite{Henrici1977:applied_computational_complex,
  Hille1976:ordinary_differential_equations,
  Teschl2012:ordinary_differential_equations}.
\section{Quasimodular vectors}\label{sec:quasi-modul-vect}
In this section we will use the transformation behavior of quasimodular forms
to define vector valued functions that encode this transformation
behavior. This has some analogy to vector valued modular forms as studied in
\cite{Franc_Mason2016:hypergeometric_series_modular,
  Franc_Mason2014:fourier_coefficients_vector,
  Knopp_Mason2012:logarithmic_vector_valued,
  Mason2012:fourier_coefficients_vector,
  Knopp_Mason2011:logarithmic_vector_valued,
  Marks_Mason2010:structure_module_vector,
  Mason2007:vector_valued_modular,Knopp_Mason2003:vector_valued_modular}, but
also exhibits some differences.

Let $f$ be a holomorphic quasimodular form of weight $w$ and depth $s$. Then
$f$ can be written as
\begin{equation}\label{eq:quasif}
    f(z)=\sum_{\ell=0}^sE_2(z)^\ell h_\ell(z),
\end{equation}
where $h_\ell$ ($\ell=0,\ldots,s$) are modular forms of weight
$w-2\ell$. Define quasimodular forms $g_\ell$ ($\ell=0,\ldots,s$) of weight
$w-2\ell$ and depth $s-\ell$ by
\begin{equation}\label{eq:quasig}
    \binom s\ell g_\ell(z)=\left(\frac6{\pi i}\right)^\ell
    \sum_{m=0}^{s-\ell}\binom{\ell+m}mE_2(z)^mh_{\ell+m}(z);
\end{equation}
notice that $f=g_0$.  Then we have
\begin{equation}
  \label{eq:quasiST}
  f(Tz)=f(z)\quad\text{and }
  z^{-w}f(Sz)= \sum_{\ell=0}^s\binom s\ell \frac1{z^\ell}g_\ell(z),
\end{equation}
which follows from \eqref{eq:E2S}.

\begin{definition}\label{def:quasi-vec}
Let $f$ be a quasimodular form of weight $w$ and depth $s\leq r$ given by
\eqref{eq:quasif}. Let then the forms $g_\ell$ ($\ell=0,\ldots,s$) be given  by
\eqref{eq:quasig}. Use these to define
  \begin{equation*}
    f_k(z)=\sum_{\ell=0}^{\min(k,s)}\binom k\ell z^{k-\ell} g_{\ell}(z)
  \end{equation*}
  for $k=0,\ldots,r$.  A holomorphic vector valued
  function $\vec{F}:\HH\to\C^{r+1}$ is called quasimodular, if there is a
  quasimodular form $f$ such that $\vec{F}$ is given by
  \begin{equation*}
    \vec{F}(z)=(f_0(z),f_1(z),\ldots,f_r(z))^\trans.
  \end{equation*}
  If $s<r$, we call $\vec{F}$ \emph{degenerate}.
\end{definition}
\begin{proposition}\label{prop:transf}
  Let $\vec{F}=(f_0,\ldots,f_r)^\trans$ be a holomorphic vector function on
  $\HH$.  Then $\vec{F}$ is quasimodular, if and only if it has the following
  behavior under the generators $S$ and $T$ of $\Gamma$:
  \begin{align}\label{eq:F-transS}
  z^{r-w}\vec{F}(Sz)&=\rho(S)\vec{F}(z)\\
  \vec{F}(Tz)&=\rho(T)\vec{F}(z)\label{eq:F-transT}
\end{align}
with
\begin{equation}\label{eq:rhoS}
  \rho(S)=
  \begin{pmatrix}
    0&0&\ldots&0&1\\
    0&0&\ldots&-1&0\\
    0&0&\iddots&0&0\\
    0&\iddots&&0&0\\
    (-1)^r&0&\ldots&0&0
  \end{pmatrix}
\end{equation}
and
\begin{equation}\label{eq:rhoT}
  \rho(T)=
  \begin{pmatrix}
    1&0&0&\ldots&0&0\\
    1&1&0&\ldots&0&0\\
    1&2&1&\ldots&0&0\\
    1&3&3&\ddots&0&0\\
    \ldots&\ldots&\ldots&\ldots&\ddots&0\\
    1&\binom r1&\binom r2&\ldots&\binom r{r-1}&1
  \end{pmatrix}.
\end{equation}
\end{proposition}

\begin{proof}
  Let
  \begin{equation*}
    f(z)=\sum_{\ell=0}^rE_2^\ell h_\ell
  \end{equation*}
  be a quasimodular form of weight $w$ and depth $s\leq r$; the case $s<r$ is
  included by setting $h_\ell=0$ for $\ell>s$. Then $h_\ell$
  ($\ell=0,\ldots,r$) is a modular form of weight $w-2\ell$. The transformation
  behavior of $f$ under $S$ and $T$ is given by \eqref{eq:quasiST}.

  The forms $g_\ell$ given by \eqref{eq:quasig} transform under $S$ by
\begin{equation*}
  z^{2\ell-w}g_\ell(Sz)=\sum_{m=0}^{r-\ell}\binom{r-\ell}m\frac1{z^m}g_{m+\ell}(z).
\end{equation*}

Using this we obtain
\begin{align*}
  z^{r-w}f_k(Sz)&=\sum_{\ell=0}^k\left(-\frac1z\right)^\ell
  z^{r-w}g_{k-\ell}(Sz)\\
  &= (-1)^k\sum_{m=0}^rz^{r-k-m}g_m(z)\sum_{\ell=0}^m(-1)^\ell
  \binom k\ell\binom{r-\ell}{r-m}.
\end{align*}
The inner sum equals $\binom{r-k}m$, which gives
\begin{equation*}
 z^{r-w}f_k(Sz)=(-1)^k\sum_{m=0}^{r-k}\binom{r-k}mz^{r-k-m}g_m(z)=(-1)^kf_{r-k}(z).
\end{equation*}

Similarly, we have for the transformation behavior under $T$
\begin{align*}
  f_k(Tz)&=\sum_{\ell=0}^k\sum_{m=0}^\ell\binom\ell m\binom k\ell
  z^mg_{k-\ell}(z)\\
  &=  \sum_{p=0}^k\binom kp\sum_{m=0}^p\binom pm z^mg_{p-m}(z)=
   \sum_{p=0}^k\binom kp f_p(z).
 \end{align*}
 
 Assume now that $\vec{F}=(f_0,\ldots,f_r)^\trans$ is an $(r+1)$-dimensional
 vector valued function satisfying \eqref{eq:F-transS} and
 \eqref{eq:F-transT}. Then $f_0$ is $T$-invariant, thus admits a power series
 representation in $q$, which we denote by $g_0$ for simplifying the notation
 in the following argument. By \eqref{eq:rhoT} the second coordinate $f_1$
 satisfies
  \begin{equation*}
    f_1(Tz)=f_1(z)+f_0(z);
  \end{equation*}
  we consider the  function
  \begin{equation*}
    g_1(z)=f_1(z)-zf_0(z),
  \end{equation*}
  which is $T$-invariant. Thus we can write
  \begin{equation*}
    f_1(z)=zg_0(z)+g_1(z),
  \end{equation*}
  where $g_1$ is a power series in $q$. Assume now by induction that we have
  already shown that
  \begin{equation*}
    f_m(z)=\sum_{\ell=0}^m\binom m\ell z^\ell g_{m-\ell}(z)
  \end{equation*}
for $0\leq m<k$  where each of the functions $g_0,\ldots,g_{k-1}$ is a power
series in $q$. We define the function
\begin{equation*}
  g_k(z)=f_k(z)-\sum_{\ell=1}^k\binom k\ell z^\ell g_{k-\ell}(z).
\end{equation*}
Then we have
\begin{align*}
  g_k(Tz)&=\sum_{\ell=0}^k\binom k\ell f_\ell(z)-
  \sum_{\ell=1}^k\binom k\ell\sum_{m=0}^\ell z^m g_{k-\ell}(z)\\
  &=  f_k(z)+\!\sum_{\ell=0}^{k-1}\binom k\ell\sum_{m=0}^\ell\binom \ell m
  z^mg_{\ell-m}(z)-\!
  \sum_{\ell=1}^k\binom k\ell\sum_{m=0}^\ell \binom \ell mz^m g_{k-\ell}(z).
\end{align*}
A similar rearrangement as before then gives $g_k(Tz)=g_k(z)$, which shows that
$g_k$ can be expressed as a power series in $q$. Summing up,
it follows from \eqref{eq:F-transT} that there are power series in $q$,
$g_0,\ldots,g_r$, such that each of the functions $f_k$, $k=0,\ldots,r$ can be
expressed as
\begin{equation*}
  f_k(z)=\sum_{\ell=0}^k\binom k\ell z^\ell g_{k-\ell}(z).
\end{equation*}

Assume now that in addition \eqref{eq:F-transS} holds. Then we have
\begin{equation*}
  z^{r-w}f_0(Sz)=f_r(z)=\sum_{\ell=0}^r\binom r\ell z^\ell g_{r-\ell}(z).
\end{equation*}
Now we have (recall that $f_0=g_0$)
\begin{align*}
  z^{r-w}f_1(Sz)&=z^{r-w}\left(\left(-\frac1z\right)g_0(Sz)+g_1(Sz)\right)\\
  &=  -\frac1z\sum_{\ell=0}^r\binom r\ell z^\ell g_{r-\ell}(z)+z^{r-w}g_1(Sz)\\
  &=  -f_{r-1}(z)=-\sum_{\ell=0}^{r-1}\binom{r-1}\ell z^\ell g_{r-1-\ell}(z),
\end{align*}
from which we derive
\begin{align*}
  z^{r-w}g_1(Sz)&=\sum_{\ell=0}^r\binom r\ell z^{\ell-1}g_{r-\ell}(z)-
  \sum_{\ell=0}^{r-1}\binom{r-1}\ell z^\ell g_{r-1-\ell}(z)\\
  &=  \frac1zg_r(z)+\sum_{\ell=0}^{r-1}\left[\binom r{\ell+1}-\binom
    {r-1}\ell\right]
  z^\ell g_{r-1-\ell}(z)\\
  &=\sum_{\ell=0}^{r-1}\binom {r-1}\ell z^{\ell-1}g_{r-\ell}(z),
\end{align*}
which gives
\begin{equation*}
  z^{2-w}g_1(Sz)=\sum_{\ell=0}^{r-1}\binom{r-1}\ell\frac1{z^\ell}g_{\ell+1}(z).
\end{equation*}
Assume now by induction that we have already shown
\begin{equation*}
  z^{2m-w}g_m(Sz)=\sum_{\ell=0}^{r-m}\binom {r-m}\ell \frac1{z^\ell}g_{m+\ell}(z)
\end{equation*}
for $0\leq m<k$.

Applying $S$ to $f_k$ gives
\begin{align*}
  z^{r-w}f_k(Sz)&=z^{r-w}\sum_{\ell=0}^k\binom k\ell\left(-\frac1z\right)^\ell
  g_{k-\ell}(Sz)\\
  &=z^{r-w}g_k(Sz)+\sum_{\ell=1}^k(-1)^\ell\binom k\ell z^{r+\ell-2k}
  \sum_{m=0}^{r+\ell-k}\binom{r+\ell-k}m\frac1{z^m}g_{k-\ell+m}(z)\\
  &=z^{r-w}g_k(Sz)+\sum_{m=0}^rz^{m-k}g_{r-m}(z)
  \sum_{\ell=1}^{k}(-1)^\ell\binom k\ell\binom{r+\ell-k}m.
\end{align*}
The inner sum evaluates to
\begin{equation*}
  \sum_{\ell=1}^{k}(-1)^\ell\binom k\ell\binom{r+\ell-k}m
  =(-1)^k\binom {r-k}{m-k}-\binom{r-k}m,
\end{equation*}
which gives
\begin{align}\label{eq:f_kS}
  z^{r-w}f_k(Sz)&=
  z^{r-w}g_k(Sz)+(-1)^k\sum_{m=k}^r\binom{r-k}{m-k}z^{m-k}g_{r-m}(z)\\
   &-\sum_{m=0}^{r-k}\binom{r-k}m z^{m-k}g_{r-m}(z).
\end{align}
On the other hand we have
\begin{equation*}
  z^{r-w}f_k(Sz)=(-1)^kf_{r-k}(z)=
  (-1)^k\sum_{\ell=0}^{r-k}\binom{r-k}\ell z^\ell g_{r-k-\ell}(z),
\end{equation*}
which equals the first sum after shifting the index of summation.
Inserting this into \eqref{eq:f_kS} gives
\begin{equation*}
  z^{r-w}g_k(Sz)=\sum_{m=0}^{r-k}\binom{r-k}m z^{m-k}g_{r-m}(z).
\end{equation*}
Thus $f_0=g_0$ has the transformation behavior of a quasimodular form under
the action of $S$ and $T$. Since $S$ and $T$ generate $\Gamma$, this together
with the assumed holomorphy implies that $f_0$ is a quasimodular form.
\end{proof}
A similar reasoning was used in
\cite{Cohn_Kumar_Miller+2019:universal_optimality} in the case of $r=2$ to show
that the solutions of certain functional equations were quasimodular.

In order to get a better understanding of quasimodular vectors, we study the
modular Wronskian of a quasimodular vector $\vec{F}$:
\begin{equation}\label{eq:wronski}
  W(z)=W_{\vec{F}}(z)=\det\left(\vec{F},\partial_{w-r}\vec{F},\ldots,
  \partial_{w-r}^r\vec{F}\right).
\end{equation}
For vector valued modular forms this has been studied in
\cite{Mason2007:vector_valued_modular}. The most important property of $W$ is
its modularity.
\begin{proposition}\label{prop:wronski}
  Let $\vec{F}$ be a quasimodular vector of weight $w$ and dimension
  $r+1$. Then the corresponding modular Wronskian $W_{\vec{F}}$ is a modular
  form of weight $w(r+1)$.
\end{proposition}
\begin{proof}
 Let $\vec{F}$ be a quasimodular vector. Then $\vec{F}(Tz)=\rho(T)\vec{F}(z)$
 and therefore $W(Tz)=W(z)$.

 For the transformation behavior under $S$, we recall that by
 Lemma~\ref{lem:serre} and Proposition~\ref{prop:transf}
 \begin{equation*}
   \rho(S)\partial_{w-r}^\ell\vec{F}(z)=\partial_{w-r}^\ell z^{r-w}\vec{F}(Sz)=
   z^{r-w-2\ell}\left(\partial_{w-r}^\ell\vec{F}\right)(Sz),
 \end{equation*}
 from which we derive
 \begin{equation*}
   z^{-w(r+1)}W(Sz)=\det(\rho(S))W(z).
 \end{equation*}
 Since $\det(\rho(S))=1$, this gives the assertion.
\end{proof}
For the fundamental system of a normalized modular differential equation we
have a far more precise statement. This is the analogue to
\cite[Theorem~4.3]{Mason2007:vector_valued_modular}.
\begin{proposition}\label{prop:normal}
  Let $f_0,f_1,\ldots,f_r$ be a fundamental system of solutions of the
 normalized modular differential equation
  \begin{equation*}
    \partial_{w-r}^{r+1}f+B_4\partial_{w-r}^{r-1}f+\cdots+B_{2r+2}f=0,
  \end{equation*}
  where $B_4,B_6,\ldots,B_{2r+2}$ are modular form of respective weights
  $4,6,\ldots,\allowbreak 2r+2$. Assume further that the solutions of the
  indicial equation $\lambda_0\geq\lambda_1\geq\cdots\geq\lambda_r\geq0$ are
  all integers. Then the modular Wronskian of $f_0,\ldots,f_r$ equals
  $c\Delta^{\frac{w(r+1)}{12}}$ for some constant $c\neq0$.
\end{proposition}
\begin{proof}
  Without loss of generality, we take $f_0,f_1,\ldots,f_r$ as the solutions
  obtained by the Frobenius ansatz in this order. Notice that
  $\lambda_0+\cdots+\lambda_r=\frac{w(r+1)}{12}$. Let
  $\vec{F}=(f_0,\ldots,f_r)^\trans$. Then there exist matrices $\rho(S)$ and
  $\rho(T)$, such that
  \begin{equation*}
    \vec{F}(Tz)=\rho(T)\vec{F}(z)\quad\text{and }
    z^{r-w}\vec{F}(Sz)=\rho(S)\vec{F}(z).
  \end{equation*}
  By the construction of $\vec{F}$ from the Frobenius ansatz it follows that
  $\rho(T)$ is lower triangular with entries $1$ on the diagonal, which gives
  $\det(\rho(T))=1$. On the other hand we have
  \begin{equation*}
    \rho(S)^2=(-1)^r\id,
  \end{equation*}
  from which we derive $\det(\rho(S))^2=1$. Furthermore, we have
  \begin{equation*}
    \left(\rho(S)\rho(T)\right)^3=(-1)^r\id,
  \end{equation*}
  which implies $\det(\rho(S))^3=1$. Thus we finally have $\det(\rho(S))=1$.

  Applying Lemma~\ref{lem:serre} we obtain
  \begin{equation*}
    \partial_{w-r}^k\left(z^{r-w}\vec{F}(Sz)\right)=
    z^{r-w-2k}\left(\partial_{w-r}^k\vec{F}\right)(Sz).
  \end{equation*}
  The Wronskian
  $W=\det(\vec{F},\partial_{w-r}\vec{F},\ldots,\partial_{w-r}^r\vec{F})$ then
  satisfies
  \begin{align*}
    &z^{-w(r+1)}W(Sz)\\
    &=\det\left(z^{r-w}\vec{F}(Sz),
      z^{r-w-2}\left(\partial_{r-w}\vec{F}\right)(Sz),\ldots,
      z^{-r-w}\left(\partial_{r-w}^r\vec{F}\right)(Sz)\right)\\
    &=\det\left(\rho(S)\vec{F}(z),\rho(S)\partial_{w-r}\vec{F}(z),\ldots,
    \rho(S)\partial_{w-r}^r\vec{F}(z)\right)=\det(\rho(S))W(z)=W(z).
\end{align*}
Similarly, we have $W(Tz)=W(z)$. Thus $W$ is a modular form with weight
$w(r+1)$. From the vanishing orders of $f_i$ we obtain that $W$ vanishes to
order (at least) $\frac{w(r+1)}{12}$ at $i\infty$. Since $W(z)\neq0$ for
$z\in\HH$, this implies that $W$ has to be a non-zero multiple of
$\Delta^{\frac{w(r+1)}{12}}$, thus proving the assertion.
\end{proof}

\begin{proposition}\label{prop:solution}
  Let $f$ be a quasimodular form of depth $\leq r$ and weight $w$ and assume
  that $f$ is a solution of a differential equation \eqref{eq:diff-r}. Let
  then $g_\ell$ ($\ell=0,\ldots,r$) be given by \eqref{eq:quasig}. Then the
  functions
  \begin{equation}
    \label{eq:fundamental}
    f_\ell(z)=\sum_{m=0}^\ell\binom\ell m z^mg_{\ell-m}(z)
  \end{equation}
  form a fundamental system of \eqref{eq:diff-r}.
\end{proposition}
\begin{proof}
  Let $f=f_0$ be a solution of \eqref{eq:diff-r}. Then by $S$-invariance
  of the differential equation, also the function
  \begin{equation*}
z^{r-w}f(Sz)=\sum_{\ell=0}^r\binom r\ell z^\ell g_{r-\ell}(z)=f_r(z)
\end{equation*}
is a solution of \eqref{eq:diff-r}. By the invariance under $T$ of the
differential equation also the functions
\begin{equation*}
  f_r(T^kz)=\sum_{\ell=0}^r\binom r\ell k^{r-\ell} f_\ell(z)
\end{equation*}
are solutions for $k\in\mathbb{Z}$. Here we have used
Proposition~\ref{prop:transf}. This gives
\begin{equation*}
  \begin{pmatrix}
    f_r(z)\\[2mm]f_r(Tz)\\\vdots\\f_r(T^rz)
  \end{pmatrix}
  =
  \begin{pmatrix}
    0&0&\ldots&0&1\\
    1&1&\ldots&1&1\\
    2^{r+1}&2^r&\ldots&2&1\\
    \vdots&\vdots&\ddots&\vdots&1\\
    r^{r+1}&r^r&\ldots&r&1
  \end{pmatrix}
  \begin{pmatrix}
    \binom r0f_0(z)\\[2mm]
    \binom r1f_1(z)\\
    \vdots\\
    \binom rrf_r(z)
  \end{pmatrix},
\end{equation*}
which shows that the functions $f_\ell(z)$ ($\ell=0,\ldots,r$) can be expressed
as linear combinations of $f_r(T^kz)$ ($k=0,\ldots,r$) and therefore are again
solutions of \eqref{eq:diff-r}. Notice that the matrix is a Vandermonde matrix
and thus invertible.

Since the functions $f_\ell$ ($\ell=0,\ldots,r$) are linearly
independent they form a fundamental system of \eqref{eq:diff-r}.
\end{proof}
\section{Balanced quasimodular forms}\label{sec:balanc-quas-forms}
We start with a proposition that clarifies the possible orders of vanishing of
quasimodular vectors and the underlying quasimodular forms. This answers a
question posed in \cite{Kaneko_Koike2006:extremal_quasimodular_forms}. The
assertion of the proposition is the analogue to
\cite[Theorem~3.7 and Corollary~3.8]{Mason2007:vector_valued_modular}. A
similar inequality for the case of extremal quasimodular forms is given in
\cite{Pellarin_Nebe2019:extremal_quasi_modular}.
\begin{proposition}\label{prop:balanced}
  Let $f$ be a quasimodular form of weight $w$ and depth $r$ given by
  \begin{equation*}
    f(z)=\sum_{\ell=0}^rE_2^\ell h_\ell,
  \end{equation*}
  where $h_\ell$ ($\ell=0,\ldots,r$) are modular forms of weight $w-2\ell$. Let
  the functions $g_\ell$ given by \eqref{eq:quasig} have vanishing orders
  $\lambda_\ell<\dim\QM_{w-2\ell}^{r-\ell}$ ($\ell=0,\ldots,r$) with
  $\lambda_0\geq\lambda_1\geq\cdots\geq\lambda_r\geq0$, i.e.
  \begin{equation*}
    g_\ell(z)=q^{\lambda_\ell}\sum_{n=0}^\infty a_\ell(n)q^n,\quad\text{with }
    a_\ell(0)\neq0.
  \end{equation*}
  Then
  \begin{equation}
    \label{eq:vanish}
    \lambda_0+\ldots+\lambda_r\leq\frac{w(r+1)}{12}.
  \end{equation}
\end{proposition}
\begin{proof}
  Under the assumptions of the proposition the term in the definition of $f_k$,
  which does not carry a positive power of $z$, has vanishing order
  $\lambda_k$ and by the ordering of the exponents $\lambda_\ell$, this is the
  order of vanishing of $f_k$ at $i\infty$ (except possibly for terms
  multiplied by $z$). The same holds for the derivatives $\partial_{w-r}^\ell
  f_k$. Thus the Wronskian $W$ vanishes at least to order
  $\lambda_0+\cdots+\lambda_r$ (the terms carrying a $z$ are eliminated by the
  determinant by Proposition~\ref{prop:wronski}). Thus it can be written as
  \begin{equation*}
    W=\Delta^{\lambda_0\cdots+\lambda_r}H
  \end{equation*}
  for a holomorphic form $H$. Comparing the weights gives \eqref{eq:vanish}.
\end{proof}
\begin{definition}
  Let $f$ be a quasimodular form of weight $w$ and depth $r$; thus there are
  quasimodular forms $g_\ell$ ($\ell=0,\ldots,r$) of weights $w-2\ell$
  and depth $r-\ell$ such that
  \begin{equation*}
    z^{-w}f(Sz)=f(z)+\sum_{\ell=1}^r\binom r\ell\frac1{z^\ell}g_\ell(z).
  \end{equation*}
  The form $f$ is called \emph{balanced}, if there are non-negative integers
  \begin{equation}
    \label{eq:lambda<dim}
    \lambda_\ell<\dim\QM_{w-2\ell}^{r-\ell}\quad\text{for }\ell=0,\ldots,r
  \end{equation}
and
  \begin{equation}\label{eq:balanced}
    \lambda_0\geq\lambda_1\geq\cdots\geq\lambda_r,
  \end{equation}
  such that
  \begin{equation*}
    g_\ell(z)=q^{\lambda_\ell}\sum_{n=0}^\infty a_\ell(n)q^n
  \end{equation*}
  with $a_\ell(0)\neq0$ for $\ell=0,\ldots,r$ and furthermore
  \begin{equation}\label{eq:sumlambda}
    \lambda_0+\cdots+\lambda_r=\dim\QM_w^r-1
  \end{equation}
  holds.
\end{definition}
\begin{remark}\label{rem:balanced}
  Let $\lambda_0\geq\lambda_\geq\ldots\geq\lambda_r$ be integers satisfying
  \eqref{eq:lambda<dim} and \eqref{eq:sumlambda}.  Define a linear map
  $\Phi:\QM_w^r\to\mathbb{C}^{\dim\QM_w^r-1}$, which maps $f$ to
  $(b_0(0),\ldots,b_0(\lambda_0-1),b_1(0),\ldots,b_1(\lambda_1-1),
  \ldots,b_r(0),\ldots,b_r(\lambda_r-1))$, where $b_\ell(n)$ is given by
  \begin{equation*}
    g_\ell(z)=\sum_{n=0}^\infty b_\ell(n)q^n
  \end{equation*}
  for the forms $g_\ell$ associated to $f$ by \eqref{eq:quasig}. This map has a
  non-trivial kernel by dimension considerations. Thus for $w$ large enough
  quasimodular forms $f$ of depth $r$ exist such that the corresponding forms
  $g_\ell$ have \emph{at least} vanishing orders $\lambda_\ell$ at $i\infty$
  ($\ell=0,\ldots,r$) under the restrictions \eqref{eq:lambda<dim},
  \eqref{eq:balanced}, and \eqref{eq:sumlambda}.
\end{remark}

\begin{remark}
  Extremal quasimodular forms as studied in
  \cite{Kaneko_Koike2006:extremal_quasimodular_forms} are special cases, namely
  $\lambda_1=\cdots=\lambda_r=0$, which gives the maximum possible order of
  vanishing of a quasimodular form of weight $w$ obtained by the argument given
  in Remark~\ref{rem:balanced}.
\end{remark}

\begin{remark}\label{rem:dim}
  Notice that for $r\leq4$ and $w(r+1)\equiv0\pmod{12}$
  \begin{equation*}
    \dim\QM_w^r-1=\frac{w(r+1)}{12}.
  \end{equation*}
  Thus there is equality in \eqref{eq:vanish} for balanced forms of depth
  $\leq4$.
\end{remark}
\begin{remark}
  Notice that as opposed to the situation studied in
  \cite{Mason2007:vector_valued_modular} the assumption \eqref{eq:balanced} on
  the ordering of the vanishing orders is a restriction in our
  case. Nevertheless, this restrictive condition will be satisfied in our later
  applications.
\end{remark}
\begin{remark}\label{rem:exist}
  Combining Proposition~\ref{prop:balanced} and Remarks~\ref{rem:balanced}
  and~\ref{rem:dim} shows that balanced quasimodular forms exist for $r\leq4$
  for any choice of $\lambda_0\geq\lambda_1\geq\cdots\geq\lambda_r\geq0$
  satisfying \eqref{eq:lambda<dim} and \eqref{eq:sumlambda}.
\end{remark}
\begin{theorem}\label{thm:diffeq}
  Every balanced quasimodular form of depth $r\leq4$ and weight $w$ with
  $(r+1)w\equiv0\pmod{12}$ is a solution of a modular differential equation of
  the form
  \begin{equation}\label{eq:quasi-diff}
    \partial_{w-r}^{r+1}f+a_4E_4\partial_{w-r}^{r-1}f+\cdots+
    a_{2r+2}E_{2r+2}f=0
  \end{equation}
  with $a_4,a_6,\ldots,a_{2r+2}\in\mathbb{Q}$.
\end{theorem}
\begin{proof}
  Let $f$ be a form satisfying the assumptions of the theorem. Then choose
  $a_4,\ldots,a_{2r+2}$ so that the indicial equation of \eqref{eq:quasi-diff}
  \begin{equation}\label{eq:index}
    \begin{split}
      &\left(\lambda-\frac{w+r}{12}\right)\left(\lambda-\frac{w+r-2}{12}\right)
      \cdots\left(\lambda-\frac{w-r}{12}\right)\\
      &+a_4\left(\lambda-\frac{w+r-4}{12}\right)\cdots
      \left(\lambda-\frac{w-r}{12}\right)+\cdots +a_{2r+2}=0
    \end{split}
  \end{equation}
  has solutions $\lambda_0,\ldots,\lambda_r$ (counted with multiplicity). The
  $\lambda^r$-term comes from the first summand and has coefficient
  $-\frac{w(r+1)}{12}$, which is an integer by assumption. Thus we have
  \begin{equation*}
    \lambda_0+\cdots+\lambda_r=\frac{w(r+1)}{12}=\dim\QM_w^r-1
  \end{equation*}
  by \eqref{eq:dim-qm}.
  The coefficients
  $a_4,\ldots,a_{2r+2}$ are then all rational.

  With this choice of coefficients define the differential operator
  \begin{equation*}
    K=\partial_{w-r}^{r+1}+a_4E_4\partial_{w-r}^{r-1}+\cdots+a_{2r+2}E_{2r+2}.
  \end{equation*}
  By Lemma~\ref{lem:quasi-serre}, $\phi=Kf$ is then a quasimodular form of
  weight $w+2r+2$ and depth $\leq r$. Following Definition~\ref{def:quasi-vec}
  we define the functions $f_\ell$ ($\ell=0,\ldots,r$) from $f$ and set
  \begin{equation*}
    \phi_\ell=Kf_\ell.
  \end{equation*}
  Then we have
  \begin{equation*}
    z^{-r-w-2}\phi_\ell(Sz)=K\left(z^{r-w}f_\ell(Sz)\right)=
    (-1)^\ell Kf_{r-\ell}(z)=(-1)^\ell \phi_{r-\ell}(z)
  \end{equation*}
  and
  \begin{equation*}
    \phi_\ell(Tz)=K\left(f_\ell(Tz)\right)=\sum_{k=0}^\ell\binom\ell k Kf_k(z)=
    \sum_{k=0}^\ell\binom\ell k \phi_k(z),
  \end{equation*}
  which shows that $(\phi_0,\ldots,\phi_r)^\trans$ satisfies
  \eqref{eq:F-transS} and \eqref{eq:F-transT} and is thus a quasimodular vector
  of weight $w+2r+2$.

From the fact that the indicial equation of $K$ has the roots
$\lambda_0,\ldots,\lambda_r$ it follows that
\begin{equation*}
  \phi_\ell(z)=Kf_\ell(z)=\mathcal{O}(z^{\mu-1} q^{\lambda_\ell+1}),
\end{equation*}
where $\mu$ is the multiplicity of $\lambda_\ell$. Thus $\phi$ is a
quasimodular form of weight $w+2r+2$ and depth $\leq r$, such that the
corresponding orders of vanishing sum up to
\begin{equation*}
  (\lambda_0+1)+\cdots+(\lambda_r+1)=\frac{(w+12)(r+1)}{12}
  >\frac{(w+2r+2)(r+1)}{12}.
\end{equation*}
By Proposition~\ref{prop:balanced} this shows that $\phi$ has to vanish
identically and $f$ is a solution of $Kf=0$.
\end{proof}
\begin{remark}
  Since for every $r$ the sum of the solutions of the indicial equation equals
  $\frac{w(r+1)}{12}$, which is larger than the dimension of the space
  $\dim\QM_w^r$ for $r\geq5$, the assertion of the theorem is false for
  $r\geq5$: no quasimodular form of depth $\geq5$ is the solution of a
  normalized differential equation of the form \eqref{eq:quasi-diff}.
\end{remark}
\begin{remark}
  The functions $f,f_1,\ldots,f_r$ are the functions that would be obtained by
  solving \eqref{eq:quasi-diff} using the Frobenius ansatz (see
  \cite{Henrici1977:applied_computational_complex,
    Teschl2012:ordinary_differential_equations,
    Hille1976:ordinary_differential_equations}) in this order.
\end{remark}
Theorem~\ref{thm:diffeq} has an obvious converse.
\begin{theorem}\label{thm:solutions}
  Let $r\leq4$ be a natural number, $w$ such that $w(r+1)\equiv0\pmod{12}$, and
  $a_4,a_6,\ldots,a_{2r+2}$ be rational numbers such that the equation
  \eqref{eq:index} has only non-negative integer solutions
  $\lambda_0\geq\lambda_1\geq\cdots\geq\lambda_r$. Then the solution of the
  differential equation
  \begin{equation}\label{eq:diffeq}
    \partial_{w-r}^{r+1}f+a_4E_4\partial_{w-r}^{r-1}f+\cdots+a_{2r+2}E_{2r+2}f=0
  \end{equation}
  with $q$-expansion
  \begin{equation}\label{eq:quasi-solution}
    f(z)=q^{\lambda_0}\sum_{n=0}^\infty a_0(n)q^n,\quad a_0(0)=1
  \end{equation}
  is a balanced quasimodular form of weight $w$ and depth $r$, if at least one
  of the inequalities in \eqref{eq:balanced} is strict. If
  $\lambda_0=\cdots=\lambda_r$ (which implies that $w\equiv0\pmod{12}$), then
  the functions $z^\ell\Delta^{\frac w{12}}$ ($\ell=0,\ldots,r$) form a
  fundamental system of \eqref{eq:diffeq}.
\end{theorem}
\begin{proof}
  Define $g$ to be the balanced quasimodular form with exponents
  $\lambda_0\geq\cdots\geq\lambda_r$. Such a form exists by
  Remark~\ref{rem:exist}. Then by Theorem~\ref{thm:diffeq} $g$ is
  the solution of the differential equation \eqref{eq:quasi-diff}. Since the
  exponents uniquely define the coefficients $a_4,\ldots,a_{2r+2}$ this is the
  same equation as \eqref{eq:diffeq}. Then $g$ is the solution of
  \eqref{eq:diffeq} characterized by vanishing order $\lambda_0$ at
  $i\infty$. Thus $g=cf$ for some $c\in\mathbb{C}\setminus\{0\}$ and $f$ is a
  quasimodular form. From Proposition~\ref{prop:solution} we know that
  $f_0=f,f_1,\ldots,f_r$ (with the notation of Definition~\ref{def:quasi-vec})
  form a fundamental system of \eqref{eq:quasi-diff}. This shows the first
  assertion of the Theorem.
  
  It only remains to show that $\Delta^{\frac w{12}}$ is the solution of a
  normalized quasimodular differential equation for any $r\geq0$ with
  coefficients chosen so that $\lambda_0=\frac w{12}$ is an $(r+1)$-fold zero of
  the indicial equation. For this purpose we observe that
  \begin{equation*}
    \partial_{w-r}^k\Delta^{\frac w{12}}=\Delta^{\frac w{12}}Q_k^{(r)},
  \end{equation*}
  where $Q_k^{(r)}$ is a quasimodular form of weight $2k$ and depth $\leq
  k$. These forms satisfy the recurrence relation
  \begin{equation*}
    Q_k^{(r)}=\frac{r+1-k}{12}E_2Q_{k-1}^{(r)}+\partial_{k-1}Q_{k-1}^{(r)}
  \end{equation*}
  with initial condition $Q_0^{(k)}=1$. This recursion shows that the depth
  increases with $k$, except for $k=r+1$, where the first
  term vanishes. Thus $Q_{r+1}^{(r)}$ has depth $\leq r$; indeed it has depth
  $r-1$, since there is no quasimodular form of weight $2r+2$ and depth
  $r$. By successively subtracting the highest power of $E_2$ we obtain modular
  forms $B_4,B_6,\ldots,B_{2r}$ such that
  \begin{equation*}
    Q_{r+1}^{(r)}+B_4Q_{r-1}^{(r)}+\cdots+B_{2r}Q_1^{(r)}+B_{2r+2}Q_0^{(r)}=0.
  \end{equation*}
  Now take $\mathbf{B}=(1,0,B_4,\ldots,B_{2r})$.  The corresponding modular
  differential operator $K_{\mathbf{B}}$ then annihilates
  $\Delta^{\frac w{12}}$. A fundamental system of solutions of
  $K_{\mathbf{B}}f=0$ is given by
  \begin{equation*}
    \Delta^{\frac w{12}}, z\Delta^{\frac w{12}},\ldots, z^{r}\Delta^{\frac w{12}},
  \end{equation*}
  which can be derived from the fact that
  $z^{r-w}\Delta^{\frac w{12}}(Sz)=z^r\Delta^{\frac w{12}}$ is also a solution
  by the invariance properties of $K_{\mathbf{B}}$. The other elements of the
  fundamental system can be found by applying $T$ and taking differences.
\end{proof}
\begin{remark}
  The proof shows that for every $r\geq0$ there is a linear differential
  equation \eqref{eq:diffeq} such that the functions
  $z^\ell\Delta^{\frac w{12}}$, $\ell=0,\ldots,r$ form a fundamental system of
  solutions.
\end{remark}
\section{Modular differential equations for quasimodular
  forms of depth $\leq4$}\label{sec:modul-diff-equat}
In this section we discuss the consequences of Theorems~\ref{thm:diffeq}
and~\ref{thm:solutions} for finding balanced quasimodular forms of depths
$r\leq4$ also for weights $w$, which do not satisfy
$w(r+1)\equiv0\pmod{12}$. We include the case $r=0$ for completeness, even if
the results are rather trivial (see Table~\ref{tab:depth0}).

In this section we will use the notation
\begin{equation}
  \label{eq:ominus}
  (\lambda_0,\lambda_1,\ldots,\lambda_r)\ominus1
\end{equation}
for to denote a new set of exponents satisfying the order condition
\eqref{eq:balanced}, but with one exponent diminished by $1$.

\begin{table}[h]
  \renewcommand{\arraystretch}{1.3}
  \centering
    \caption{The forms and differential equations for $r=0$}
  \label{tab:depth0}
  \begin{tabular}[h]{|l|l|r|}
    \hline
    $w$&$f$&differential equation\\[1mm]\hline
    $0\pmod{12}$&$\Delta^{\frac w{12}}$&$\partial_wf=0$\\[1mm]\hline
    $2\pmod{12}$&$E_4^2E_6\Delta^{\frac {w-14}{12}}$&
    $E_4E_6\partial_wf+\frac16(3E_4^3+4E_6^2)f=0$\\[1mm]\hline
    $4\pmod{12}$&$E_4\Delta^{\frac {w-4}{12}}$&
    $E_4\partial_wf+\frac13E_6f=0$\\[1mm]\hline
    $6\pmod{12}$&$E_6\Delta^{\frac {w-6}{12}}$&
    $E_6\partial_wf+\frac12E_4^2f=0$\\[1mm]\hline
    $8\pmod{12}$&$E_4^2\Delta^{\frac {w-8}{12}}$&
    $E_4\partial_wf+\frac23E_6f=0$\\[1mm]\hline
    $10\pmod{12}$&$E_4E_6\Delta^{\frac {w-10}{12}}$&
    $E_4E_6\partial_wf+\frac16(3E_4^3+2E_6^2)f=0$\\[1mm]\hline
  \end{tabular}
\end{table}

\subsection{Depth $1$}
In this case the dimension of the space of quasimodular forms is given by
\begin{equation}
  \label{eq:dim-depth1}
  \dim\QM_w^1=\left\lfloor\frac w6\right\rfloor+1.
\end{equation}
\begin{theorem}\label{thm:depth1}
  Let $w\equiv a\pmod6$ ($a=0,2,4$) and let
  $\frac{w-a}{12}\leq\lambda\leq\frac{w-a}6$. Let
  \begin{equation*}
    f(z)=q^\lambda\sum_{n=0}^\infty a(n)q^n.
  \end{equation*}
  \begin{description}\setlength{\itemindent}{-\leftmargin}
  \item[\underline{$a=0$}] If $f$ is a solution of
    \begin{equation}
      \label{eq:depth1-0}
    \partial_{w-1}^2f_w-\frac{(12\lambda-w-1)(12\lambda-w+1)}{144}
    E_4f_w=0,
  \end{equation}
  then $f$ is a balanced quasimodular form of weight $w$ and depth $1$, if
  $\lambda>\frac w{12}$. If $\lambda=\frac w{12}$, the functions
  $\Delta^{\frac w{12}}$ and $z\Delta^{\frac w{12}}$ form a fundamental system
  of \eqref{eq:depth1-0}.
\item[\underline{$a=2$}] If $f$ is a solution of
  \begin{equation}
    \label{eq:depth1-2}
    E_4\partial_{w-1}^2f_w+\frac13E_6\partial_{w-1}f_w-
    \frac{(12\lambda-w+1)(12\lambda-w+3)}{144}E_4^2f_w=0
  \end{equation}
   then $f$ is a balanced quasimodular form of weight $w$ and depth $1$, if
   $\lambda>\frac{w-2}{12}$. If $\lambda=\frac {w-2}{12}$, the functions
   $E_2\Delta^{\frac {w-2}{12}}$ and 
  $zE_2\Delta^{\frac {w-2}{12}}+\frac6{\pi i}\Delta^{\frac{w-2}{12}}$ form a
  fundamental system of \eqref{eq:depth1-2}.
 \item[\underline{$a=4$}] If $f$ is a solution of
  \begin{equation}
    \label{eq:depth1-4}
    \begin{split}
      &E_4^2\partial_{w-1}^2f_w+\frac23E_4E_6\partial_{w-1}f_w\\
      &-
      \left(\left(\frac{(12\lambda-w+3)(12\lambda-w+5)}{144}-
          \frac1{18}\right)E_4^3+384\Delta\right)f_w=0
      \end{split}
  \end{equation}
  then $f$ is a balanced quasimodular form of weight $w$ and depth $1$, if
  $\lambda>\frac{w-4}{12}$. If $\lambda=\frac {w-4}{12}$, the functions
  $E_4\Delta^{\frac {w-4}{12}}$ and
  $zE_4\Delta^{\frac {w-4}{12}}$ form a fundamental system of \eqref{eq:depth1-4}.
  \end{description}
\end{theorem}
\begin{proof}
For $a=0$ this is the statement of Theorem~\ref{thm:solutions}.

  For $a=2$, we take $g$ to be solution of \eqref{eq:depth1-0} for $w-2$. Then
  $f=\partial_{w-3}g$ is a solution of \eqref{eq:depth1-2}. Similarly, for
  $a=4$ we take $g$ to be a solution of \eqref{eq:depth1-0} for $w-4$. Then
  $E_4g$ is a solution of \eqref{eq:depth1-4}. In both cases, neither
  application of the Serre derivative, nor multiplication by $E_4$ change the
  vanishing orders. Thus the resulting forms are still balanced.
\end{proof}

\subsection{Depth $2$}
In this case the dimension of the space $\QM_w^2$ is given by
\begin{equation}\label{eq:dim-depth2}
  \dim\QM_w^2=\left\lfloor\frac w4\right\rfloor+1.
\end{equation}
\begin{theorem}\label{thm:depth2}
  Let $w\equiv a\pmod4$ ($a=0,2$) and let
  \begin{equation}\label{eq:depth2-lambda}
    \lambda_0\geq\lambda_1\geq\lambda_2
  \end{equation}
  be positive integers with $\lambda_0+\lambda_1+\lambda_2=\frac{w-a}4$ and set
  \begin{align*}
  A&=\frac1{144}\left(4-3(w-a)^2\right)+\lambda_0\lambda_1+\lambda_0\lambda_2+
    \lambda_1\lambda_2\\
  B&=-\left(\lambda_0-\frac{w-a-2}{12}\right)\left(\lambda_1-
    \frac{w-a-2}{12}\right)
  \left(\lambda_2-\frac{w-a-2}{12}\right)
\end{align*}
  Let
  \begin{equation*}
    f(z)=q^{\lambda_0}\sum_{n=0}^\infty a(n)q^n.
  \end{equation*}
  \begin{description}\setlength{\itemindent}{-\leftmargin}
  \item[\underline{$a=0$}] If $f$ is a solution of
    \begin{equation}\label{eq:depth2-0}
  \partial_{w-2}^3f+A E_4\partial_{w-2}f+BE_6f=0
\end{equation}
then $f$ is a balanced quasimodular form of weight $w$ and depth $2$, if there
is at least one strict inequality in \eqref{eq:depth2-lambda}. If
$\lambda_0=\lambda_1=\lambda_2=\frac w{12}$ the functions
$\Delta^{\frac w{12}}$, $z\Delta^{\frac w{12}}$, and $z^2\Delta^{\frac w{12}}$
form a fundamental system of \eqref{eq:depth2-0}.
\item[\underline{$a=2$}] If $f$ is a solution of
\begin{equation}
  \label{eq:depth2-2}
  E_6\partial_{w-2}^3f+\frac12E_4^2\partial_{w-2}^2f+AE_4E_6\partial_{w-2}f+
  \left(\frac12AE_4^3+\frac13(3B-A)E_6^2\right)f=0
\end{equation}
then $f$ is a balanced quasimodular form of weight $w$ and depth $2$, if there
is at least one strict inequality in \eqref{eq:depth2-lambda}. If
$\lambda_0=\lambda_1=\lambda_2=\frac {w-2}{12}$, the functions
$E_2\Delta^{\frac {w-2}{12}}$,
$zE_2\Delta^{\frac{w-2}{12}}+\frac3{\pi i}\Delta^{\frac{w-2}{12}}$, and
$z^2E_2\Delta^{\frac {w-2}{12}}+\frac{6z}{\pi i}\Delta^{\frac{w-2}{12}}$ form a
fundamental system of \eqref{eq:depth2-2}.
\end{description}
\end{theorem}
\begin{proof}
  The case $a=0$ is covered by Theorem~\ref{thm:solutions}. For $a=2$ we
  observe that if $g$ is a solution of \eqref{eq:depth2-0} for $w-2$, then
  $\partial_{w-4}g$ is a solution of \eqref{eq:depth2-2}.
\end{proof}

\subsection{Depth $3$}
In this case the dimension of the space $\QM_w^3$ is given by
\begin{equation}\label{eq:dim-depth3}
  \dim\QM_w^3=\left\lfloor\frac w3\right\rfloor+1.
\end{equation}
We denote
\begin{align*}
  \sigma_2&=\lambda_0\lambda_1+\cdots+\lambda_2\lambda_3\\
  \sigma_3&=\lambda_0\lambda_1\lambda_2+\cdots+\lambda_1\lambda_2\lambda_3
\end{align*}
the elementary symmetric functions in $\lambda_0,\ldots,\lambda_3$. Using this
notation we set
\begin{align*}
  A_0&=-\frac{w^2}{24}+\sigma_2+\frac5{72}\\
  B_0&=-\frac{w^3}{216}+\frac{w^2}{72}+
  \frac {w-2}6\sigma_2-
  \sigma_3-\frac5{216}\\
  C_{0,2,4}&=
\left(\lambda_0-\frac{w-3}{12}\right)\left(\lambda_1-\frac{w-3}{12}\right)
\left(\lambda_2-\frac{w-3}{12}\right)\left(\lambda_3-\frac{w-3}{12}\right),
\end{align*}
and
\begin{align*}
  A_2&=-\frac{(w-2)^2}{24}+\sigma_2+\frac5{72}\\
  B_2&=-\frac{(w-2)^3}{216}+\frac{w-2}6\sigma_2-\sigma_3\\
D_2&=\frac{16}3(w-2)^3-16(w-2)^2+\frac{80}{3}-192(w-4)\sigma_2
+1152\sigma_3,
\end{align*}
and
\begin{align*}
  A_4&=-\frac{(w-1)^2}{24}+\frac1{36}+\sigma_2\\
  B_4&=-\frac{2w^3-9w^2+12w-3}{432}-\sigma_3+\frac{w-2}6\sigma_2\\
  D_4&=\frac43\left(2w^3-9w^2+12w-3\right)-96(w-2)\sigma_2+576\sigma_3.
\end{align*}
\begin{theorem}\label{thm:depth3}
  Let $w\equiv 0,2,4\pmod6$ and let
  \begin{equation}\label{eq:depth3-lambda}
    \lambda_0\geq\lambda_1\geq\lambda_2\geq\lambda_3
  \end{equation}
  be positive integers with
  $\lambda_0+\lambda_1+\lambda_2+\lambda_3= \lfloor\frac w3\rfloor$ and not all
  equal. Let $A_0,B_0,C_0$, $A_2,\ldots,D_2$, and $A_4,\ldots,D_4$ be given as
  above and let
  \begin{equation}\label{eq:fz-frobenius}
    f(z)=q^{\lambda_0}\sum_{n=0}^\infty a(n)q^n.
  \end{equation}
  \begin{description}\setlength{\itemindent}{-\leftmargin}
  \item[\underline{$w\equiv0\pmod6$}] If $f$ is a solution of
    \begin{equation}
      \label{eq:depth3-0}
      \partial_{w-3}^4f+A_0 E_4\partial_{w-3}^2f+B_0 E_6\partial_{w-3}f+C_0E_4^2f=0,
    \end{equation}
    then $f$ is a balanced quasimodular form of weight $w$ and depth $3$.
    \item[\underline{$w\equiv2\pmod6$}] If $f$ is a solution of
    \begin{equation}
      \label{eq:depth3-2}
      \begin{split}
        E_4\partial_{w-3}^4f&+\frac23E_6\partial_{w-3}^3f+A_2E_4^2\partial_{w-3}^2f
        +B_2E_4E_6\partial_{w-3}f\\ &+ \left(C_2 E_4^3+D_2\Delta\right)f=0,
      \end{split}   \end{equation}
    then $f$ is a balanced quasimodular form of weight $w$ and depth $3$.
  \item[\underline{$w\equiv4\pmod6$}] If $f$ is a solution of
    \begin{equation}
      \label{eq:depth3-4}
      \begin{split}
        E_4\partial_{w-3}^4f&+\frac13E_6\partial_{w-3}^3f+
        A_4E_4^2\partial_{w-3}^2f
        +B_4E_4E_6\partial_{w-3}f\\
        &+ \left(C_4 E_4^3+D_4\Delta\right)f=0,
      \end{split}
    \end{equation}
    then $f$ is a balanced quasimodular form of weight $w$ and depth $3$.
  \end{description}
  \end{theorem}
  \begin{proof}
    For $w\equiv0\pmod6$ this is the assertion of
    Theorem~\ref{thm:solutions}. For $w\equiv2\pmod6$ let $g$ be the solution
    of the form \eqref{eq:fz-frobenius} of \eqref{eq:depth3-0} for $w-2$. Then
    $\partial_{w-5}g$ is a balanced quasimodular form of weight $w$ and depth
    $3$ satisfying \eqref{eq:depth3-2}.

    For $w\equiv4\pmod6$ we have to take into account that
    $\lfloor\frac w3\rfloor=\lfloor\frac{w-4}3\rfloor+1$, which means that the
    total order of vanishing for a balanced quasimodular form of weight $w$ is
    one higher than that of a form of weight $w-4$. We take $g$ as the solution
    of \eqref{eq:depth3-0} for the exponents
    $(\lambda_0,\ldots,\lambda_3)\ominus1$, where we call $\lambda$ the
    exponent that has been decreased. We then set
    \begin{equation}\label{eq:up4}
      f=\partial_{w-7}^2g-\left(\lambda-\frac{w+5}{12}\right)
      \left(\lambda-\frac{w+7}{12}\right)E_4g.
    \end{equation}
    Then $f$ is a balanced quasimodular form of weight $w$. Notice that $g$ has
    order $\lambda-1$ for the corresponding vanishing order, whereas $f$ has
    again $\lambda$ by the choice of the operator applied to $g$. Then $f$ is a
    solution of \eqref{eq:depth3-4}.
  \end{proof}
\subsection{Depth $4$}
In this case the dimension of the space $\QM_w^4$ is given by
\begin{equation}\label{eq:dim-depth4}
  \dim\QM_w^4=\left\lfloor\frac {5w}{12}\right\rfloor+
  \begin{cases}
    1&\text{if }w\equiv0,2,4,6,8\pmod{12}\\
    0&\text{if }w\equiv10\pmod{12}.
  \end{cases}
\end{equation}
In principle, a similar theorem to Theorems~\ref{thm:depth1}--\ref{thm:depth3}
could be given. For depth $4$ it would give six modular differential equations
according to the residue classes $0,2,4,6,8,10\pmod{12}$. The equation for
$w\equiv\pmod{12}$ can just be determined from $\lambda_0,\ldots,\lambda_4$ by
requiring that these are the solutions of the indicial equation. For
$w\equiv2\pmod{12}$ the balanced quasimodular form of weight $w$ can be
obtained from the form of weight $w-2$ by applying $\partial_{w-6}$. Similarly,
for $w\equiv4\pmod{12}$ the balanced form of weight $w$ can be obtained from
the form of weight $w-4$ by a similar operator as \eqref{eq:up4}. For
$w\equiv6,8\pmod{12}$ operators of orders $3$ and $4$ have to be used, which
are determined by two respectively three roots of their indicial equation. For
$w\equiv10\pmod{12}$ the form can again be obtained by applying
$\partial_{w-6}$ to the form obtained for $w\equiv8\pmod{12}$.


\section{Recursions for extremal quasimodular forms
  and a conjecture of Kaneko  and Koike}
\label{sec:extr-quas-forms}
In this section we use the ideas developed so far to obtain differential
recursions for extremal quasimodular forms as introduced and studied by
M.~Kaneko and M.~Koike in
\cite{Kaneko_Koike2006:extremal_quasimodular_forms}. In
\cite{Kaneko_Koike2006:extremal_quasimodular_forms} a conjecture about the
possible prime divisors of the denominators of the Fourier coefficients of
normalized extremal quasimodular forms. Recently, this conjecture was proved by
F.~Pellarin and G.~Nebe \cite{Pellarin_Nebe2019:extremal_quasi_modular} for
extremal quasimodular forms of depth $1$ and weight $\equiv0\pmod6$. Later,
A.~Mono \cite{Mono2020:conjecture_kaneko_koike} extended their proof for
extremal quasimodular forms of depth $1$ and weight $\equiv2,4\pmod6$. We will
provide a proof for quasimodular forms of depths $\leq4$ as a consequence of
the explicit form of the differential recursions.

We adopt the notation used in
\cite{Kaneko_Koike2003:modular_forms_hypergeometric} that an equation number
with subscript $w+a$ means that the parameter $w$ in this equation is replaced
by $w+a$. Also, throughout this section we take $f_w$ to denote a normalized
form of weight $w$; recall that we call a form normalized, if its leading
coefficient equals $1$.
\subsection{Depth $1$}
As stated in \cite{Kaneko_Koike2006:extremal_quasimodular_forms,
  Kaneko_Koike2003:modular_forms_hypergeometric,
  Yamashita2010:construction_extremal_quasimodular} extremal quasimodular forms
of depth $1$ satisfy the differential equation
\begin{equation}\label{eq:depth1}
    \partial_{w-1}^2f_w-\frac{w^2-1}{144}E_4f_w=0
  \end{equation}
  for $w\equiv0\pmod6$, which we assume throughout this subsection.  We observe
  that
\begin{equation}\label{eq:depth1-up}
  \Kup_wf_w=E_4\partial_{w-1}f_w-\frac{w+1}{12}E_6f_w
\end{equation}
is a solution of \eqref{eq:depth1}\textsubscript{$w+6$}, if $f_w$ is a solution
of \eqref{eq:depth1}\textsubscript{$w$}. This can be seen from the fact that
\eqref{eq:depth1-up} is a quasimodular form of weight $w+6$ of depth $1$ with
one order of vanishing higher than $f_w$, thus an extremal quasimodular form of
weight $w+6$.

On the other hand,
\begin{equation}
  \label{eq:depth1-down}
  \Kdown_wf_w=\frac1{\Delta}\left(E_4\partial_{w-1}f_w+
    \frac{w-1}{12}E_6f_w\right)
\end{equation}
is a solution of \eqref{eq:depth1}\textsubscript{$w-6$}: it is a quasimodular
form of weight $w-6$ vanishing to one order less than $f_w$. The holomorphy
follows from the fact that the differential operator in parenthesis applied to
constants gives $\mathcal{O}(q)$. The leading coefficient of $\Kdown_wf_w$
equals $\frac w6$, thus
\begin{equation*}
  \Kdown_wf_w=\frac w6f_{w-6}.
\end{equation*}

Furthermore, we have
\begin{equation}\label{eq:depth1-kdown-up}
  \Kdown_{w+6}\Kup_wf_w=
    12(w+1)(w+5)f_w
\end{equation}
using \eqref{eq:depth1}.  Defining $c_{w+6}$ by
\begin{equation*}
  \Kup_wf_w=c_{w+6}f_{w+6},
\end{equation*}
we have
\begin{equation*}
  \Kdown_{w+6}\Kup_wf_{w}=c_{w+6}\frac{w+6}6f_w,
\end{equation*}
which allows to compute $c_{w+6}$ from \eqref{eq:depth1-kdown-up}.
Summing up, we have proved the following proposition.
\begin{proposition}\label{prop-depth1}
  Let $(f_w)_{w\in2\N}$ ($w\geq6$) denote the sequence of normalized extremal
  quasimodular forms of depth $1$. Then for $w\equiv0\pmod6$
  \begin{equation}
    \label{eq:depth1-init}
    \begin{split}
      f_6&=\frac1{720}\left(E_2E_4-E_6\right)
    \end{split}
  \end{equation}
  and
  \begin{equation}
    \label{eq:depth1-recurr}
    f_{w+6}=\frac{w+6}{72(w+1)(w+5)}\left(E_4\partial_{w-1}f_w
      -\frac{w+1}{12}E_6f_w\right).
  \end{equation}
  Furthermore, we have
  \begin{equation}
    \label{eq:depth1-2up}
    \begin{split}
      f_{w+2}&=\frac{12}{w+1}\partial_{w-1}f_w\\
      f_{w+4}&=E_4f_w.
    \end{split}
  \end{equation}
\end{proposition}
\subsection{Depth $2$}
The extremal quasimodular forms of weight $w$ and depth $2$
satisfy the differential equation
\begin{equation}\label{eq:depth2-diff}
 \partial_{w-2}^3f-
    \frac{3w^2-4}{144}E_4\partial_{w-2}f-
 \frac{(w+1)(w-2)^2}{864}E_6f=0
\end{equation}
for $w\equiv0\pmod4$, which we assume for this subsection.
If $f_w$ satisfies \eqref{eq:depth2-diff}\textsubscript{$w$}, then
\begin{equation}
  \label{eq:diff-recurr}
  \Kup_wf_w=\frac{w(w+1)}{36}E_4f_w-\partial_{w-2}^2f_w
\end{equation}
satisfies \eqref{eq:depth2-diff}\textsubscript{$w+4$}  and
\begin{equation}
  \label{eq:diff-down}
  \Kdown_wf_w=\frac1\Delta\left(E_4\partial_{w-2}^2f_w+
    \frac{w-1}6E_6\partial_{w-2}f_w +    \frac{(w-2)^2}{144}E_4^2f_w\right)
\end{equation}
satisfies \eqref{eq:depth2-diff}\textsubscript{$w-4$}. As before this can be
seen from the fact that $\Kup_wf_w$ and $\Kdown_wf_w$ are quasimodular forms of
respective weights $w+4$ and $w-4$, which vanish to respective orders
$\frac w4+1$ and $\frac w4-1$, thus being extremal.  Notice that the indicial
equation of $\Delta\Kdown_w$ has a double root at $0$, thus it maps linear
functions in $z$ to $q\times(\text{linear functions in }z)$. In
\cite{Kaneko_Koike2006:extremal_quasimodular_forms} essentially the same
operator expressed in terms of the Rankin-Cohen bracket (see
\cite{Zagier1994:modular_forms_differential}) has been used.

As before, define $c_{w+4}$ by
\begin{equation*}
  \Kup_wf_w=c_{w+4}f_{w+4}
\end{equation*}
and observe that
\begin{equation*}
  \Kdown_wf_w=\left(\frac w4\right)^2f_{w-4}.
\end{equation*}
Furthermore, we have
\begin{equation*}
  \Kdown_{w+4}\Kup_wf_w=
    \frac13(w+1)(w+2)^2(w+3)f_w.
\end{equation*}
Putting these together gives
\begin{equation*}
  c_{w+4}=\frac{16(w+1)(w+3)(w+2)^2}{3w^2}.
\end{equation*}

Thus we have proved
\begin{proposition}\label{prop-depth2}
  Let $(f_w)_{w\in\N}$ ($w\geq4$) denote the sequence of normalized extremal
  quasimodular forms of depth $2$. Then for $w\equiv0\pmod4$
  \begin{equation}
    \label{eq:depth2-init}
      f_4=\frac1{288}\left(E_4-E_2^2\right)
  \end{equation}
  and
  \begin{equation}
    \label{eq:depth2-recurr}
    f_{w+4}=\frac{3w^2}{16(w+1)(w+2)^2(w+3)}
    \left(\frac{(w+1)w}{36}E_4f_w-\partial_{w-2}^2f_w\right).
  \end{equation}
  Furthermore, we have
  \begin{equation}
    \label{eq:depth2-2up}
    f_{w+2}=\frac6{w+1}\partial_{w-2}f_w.
  \end{equation}
\end{proposition}

\subsection{Depth $3$}\label{sec:depth-3}
For depth $3$ the extremal quasimodular forms of weight $w$ satisfy the
differential equation
\begin{equation}
  \label{eq:depth3-diff}
    \partial_{w-3}^4f-
    \frac{3w^2-5}{72}E_4\partial_{w-3}^2f
    -\frac{w^3-3w^2+5}{216}E_6\partial_{w-3}f
    -\frac{(w+1)(w-3)^3}{6912}E_4^2f=0
\end{equation}
for $w\equiv0\pmod6$, which we assume for this subsection.

We define
\begin{equation}
  \label{eq:depth3-up}
  \begin{split}
 \Kup_wf&=48(7w^2+42w+60)\partial_{w-3}^3f\\
    &- (15 w^4+96 w^3+ 151 w^2-30
    w-116)E_4\partial_{w-3}f \\&-\frac16(w+1)(9w^4+45w^3+40w^2+24w+144)E_6f    
  \end{split}
\end{equation}
Then $\Kup_wf_w$ is a solution of \eqref{eq:depth3-diff}\textsubscript{$w+6$},
if $f_w$ is a solution of \eqref{eq:depth3-diff}\textsubscript{$w$} by a
similar reasoning as before.

For the corresponding operator $\Kdown_w$ we make an ansatz
$\Kdown_w=\frac1{\Delta^2}L_w$ with
\begin{equation*}
  L_w=A_{12}\partial_{w-3}^3+A_{14}\partial_{w-3}^2+A_{16}\partial_{w-3}+A_{18}
\end{equation*}
with unknown modular forms $A_{12},\ldots,A_{18}$ of respective weights
$12,\ldots,\allowbreak 18$. This ansatz is motivated by the fact that the form
$\Kdown_wf_w$ has to have a vanishing order two less than the vanishing order
of $f_w$. Thus the factor $\frac1{\Delta^2}$. On the other hand $\Kdown_wf_w$
has to be a form of weight $w-6$, thus $L_w$ has to add a weight of $18$. We
take $L_w$ as an operator of order $3$, because for higher order the order
could be reduced by the fact that $f_w$ satisfies a differential equation of
order $4$. The ansatz given above is the general form of an operator satisfying
the described requirements.

We apply $L_w$ to the solution $f_w$ of
\eqref{eq:depth3-diff} with $f_w=q^{\frac w3}(1+\mathcal{O}(q))$ (a
quasimodular form of weight $w$ and depth $3$). Then $L_wf_w$ is a quasimodular
form of weight $w+18$. From this form we define the quasimodular forms
$\widetilde{g}_1,\widetilde{g}_2,\widetilde{g}_3$ as in
Section~\ref{sec:quasi-modul-vect}. We define the map
\begin{align*}
  \Phi:\M_{12}\oplus\cdots\oplus\M_{18}&\to\C^6\\
  (A_{12},\ldots,A_{18})&\mapsto\text{coefficients of }1\text{ and }q\text{
    of }\widetilde{g}_1,\widetilde{g}_2,\widetilde{g}_3.
\end{align*}
Since the dimension of $\M_{12}\oplus\cdots\oplus\M_{18}$ is $7$, this map has
a non-trivial kernel. We take $(A_{12},\ldots,A_{18})$ to be in this
kernel. Then  the functions
$\widetilde{g}_1,\widetilde{g}_2,\widetilde{g}_3$ vanish to order $2$, $L_wf_w$
vanishes to order $\frac w3$, thus $\frac1{\Delta^2}L_wf_w$ is a holomorphic
form of weight $w-6$ vanishing to order $\frac w3-2$ at $i\infty$, thus an
extremal quasimodular form of this weight.

This gives
\begin{equation}
  \label{eq:depth3-down}
  \begin{split}
\Kdown_wf
&=\frac1{864\Delta^2}\Bigl(864\left(\left(9 w^2-54 w+84\right) E_4^3+
      \left(7 w^2-42 w+60\right) E_6^2\right)\partial_{w-3}^3f\\
    &    +3456 (w-3)^2 (w-1)E_4^2 E_6\partial_{w-3}^2f\\
    &    + 6E_4\Bigl(\left(39 w^4-336 w^3+1099 w^2-1626 w+924\right)
    E_6^2\\
    &    \quad+3 \left(3 w^4-48 w^3+231 w^2-450 w+316\right) E_4^3\Bigr)
    \partial_{w-3}f\\
    &    +(w-3)^2 E_6\left(3 \left(3 w^3-24 w^2+64 w-56\right) E_4^3
      -\left(w^3-24 w+48\right) E_6^2\right)f\Bigr).
  \end{split}\raisetag{2cm}
\end{equation}
Then $\Kdown_wf_w$ is a holomorphic form of weight $w-6$ with vanishing order
$\frac w3-2$, which solves \eqref{eq:depth3-diff}\textsubscript{$w-6$} for a
solution $f_w$ of \eqref{eq:depth3-diff}\textsubscript{$w$}. This gives
\begin{equation}
  \label{eq:depth3-updown}
    \Kdown_{w+6}\Kup_wf_w= 5184 (w+1) (w+2)^3 (w+3)^2 (w+4)^3 (w+5) f_w.
\end{equation}

From the fact that
\begin{equation*}
  \Kdown_{w}f_w=16(w-3)^2\left(\frac w3\right)^3q^{\frac w3-2}+\cdots=
  16(w-3)^2\left(\frac w3\right)^3f_{w-6},
\end{equation*}
we obtain with $\Kup_wf_w=c_{w+6}f_{w+6}$ that
\begin{align*}
  \Kdown_{w+6}\Kup_wf_w&=16(w+3)^2\left(\frac {w+6}3\right)^3c_{w+6}f_w\\
  &=5184 (w+1) (w+2)^3 (w+3)^2 (w+4)^3 (w+5) f_w,
\end{align*}
which gives $c_{w+6}$.

Furthermore, the Frobenius ansatz for $f_w$ gives
\begin{equation*}
  f_w=q^{\frac w3}\left(1+\frac{w \left(w^2+15 w-18\right)}{(w+3)^2}q+
    \mathcal{O}(q^2)\right).
\end{equation*}
Inserting this into $\partial_{w-3}$ and
$\frac{(w+1)(3w+1)}{48}E_4f_w-\partial_{w-3}^2f_w$
gives
\begin{align*}
  \partial_{w-3}f_w&=q^{\frac w3}\left(\frac{w+1}4+\mathcal{O}(q)\right)\\
  \frac{(w+1)(3w+1)}{48}E_4f_w-\partial_{w-3}^2f_w&=
  q^{\frac w3}\!
  \left(\frac{27 (w+1) (w+2)^3}{2 (w+3)^2}q+\mathcal{O}(q^2)\right).
\end{align*}
These two forms then have respective weights $w+2$ and $w+4$ and vanishing
orders $\lfloor\frac{w+2}3\rfloor$ and $\lfloor\frac{w+4}3\rfloor$ and are thus
extremal. The leading coefficients can be read off.

Thus we have proved
\begin{proposition}\label{prop:depth3}
  Let $(f_w)_{w\in2\N}$ ($w\geq6$) denote the sequence of
  normalized extremal quasimodular forms of depth $2$. Then
  \begin{equation}
    \label{eq:depth3-init}
      f_6=\frac{5E_2^3-3E_2E_4-2E_6}{51840}\\
  \end{equation}
  and
  \begin{equation}
    \label{eq:depth3-recurr}
    f_{w+6}=\frac{(w+6)^3}{2^23^7(w+1)(w+2)^3(w+4)^3(w+5)}\Kup_wf_w
  \end{equation}
  with $\Kup_w$ given by \eqref{eq:depth3-up}.
  Furthermore, we have
  \begin{equation}
    \label{eq:depth3-2up}
    \begin{split}
      f_{w+2}&=\frac4{w+1}\partial_{w-3}f_w\\
      f_{w+4}&=\frac{2(w+3)^2}{27(w+1)(w+2)^3}
      \left(\frac{(w+1)(3w+1)}{48}E_4f_w-\partial_{w-3}^2f_w\right).
    \end{split}
  \end{equation}
\end{proposition}
\subsection{Depth $4$}\label{sec:depth-4}
For depth $4$ the extremal quasimodular forms of weight $w$ satisfy the
differential equation
\begin{equation}
  \label{eq:depth4-diff}
  \begin{split}
    &\partial_{w-4}^5f-\frac{5}{72} \left(w^2-2\right)E_4\partial_{w-4}^3f
    -\frac{5}{432} \left(w^3-3 w^2+6\right)E_6\partial_{w-4}^2f\\
    -&\frac{15 w^4-120 w^3+280 w^2-496}{20736}E_4^2\partial_{w-4}f
    -\frac{(w-4)^4 (w+1)}{62208}E_4E_6f=0
  \end{split}
\end{equation}
for $w\equiv0\pmod{12}$, which we assume for this subsection.

We define
\begin{equation}
  \label{eq:deth4-up}
  \begin{split}
    \Kup_wf&=-p_0(w)E_4\partial_{w-4}^4f+
    \frac{(w+4)^4}{12}p_1(w)E_6\partial_{w-4}^3f\\
    &+\frac1{720}p_2(w) E_4^2\partial_{w-4}^2f
    +\frac1{8640}p_3(w)
    E_4E_6\partial_{w-4}f\\
    &+\left(\frac{w+1}{25920}p_4(w)E_4^3+
      \frac{(w+1)(w+4)^4}{15}p_5(w)\Delta\right)f.
  \end{split}
\end{equation}
The polynomials $p_0(w),\ldots,p_5(w)$ are given in
\eqref{eq:depth4-p0}--\eqref{eq:depth4-p5} in
the Appendix. These polynomials are chosen so that the operator removes the
powers $q^{\frac{5w}{12}},\ldots,q^{\frac{5w}{12}+4}$ from the power series of
$f_w$. Then $\Kup_wf$ is a solution of
\eqref{eq:depth4-diff}\textsubscript{$w+12$}, if $f$ is a solution of
\eqref{eq:depth4-diff}\textsubscript{$w$}.

Similarly, we make an ansatz for an operator $\Kdown_w=\frac1{\Delta^5}L_w$
with 
\begin{equation}
  \label{eq:depth4-down}
  L_wf=
  C_{40}\partial_{w-4}^4f+C_{42}\partial_{w-4}^3f+
  C_{44}\partial_{w-4}^2f+C_{46}\partial_{w-4}f+C_{48}f,
\end{equation}
where $C_{40},\ldots,C_{48}$ are modular forms of weights $40,\ldots,48$.
This ansatz is motivated by the fact that $\Kdown_wf_w$ has to have five less
order of vanishing than $f_w$, thus the division by $\Delta^5$. The operator
should be of order $4$, because otherwise it could reduced to lower order by
the fact that $f_w$ satisfies an equation of order $5$. In order to make
$\Kdown_wf_w$ a form of weight $w-12$, $L_wf_w$ has to have weight $w+48$, and
the ansatz above gives the general form of such an operator.

First we notice that for any such operator $L_w$, $L_wf$ is a quasimodular form
of weight $w+48$ and depth $4$. Furthermore, if $f$ vanishes to some order at
$i\infty$, then $L_wf$ vanishes to at least this order. Now take $f_w$ to be a
solution of \eqref{eq:depth4-diff}\textsubscript{$w$} with
$f_w=\mathcal{O}(q^{\frac{5w}{12}})$. Then, as in
Section~\ref{sec:quasi-modul-vect} we form the quasimodular forms
$\widetilde{g}_1,\ldots,\widetilde{g}_4$ from the function $L_wf_w$ and
consider the linear map
\begin{align*}
  \Phi:\M_{40}\oplus\cdots\oplus\M_{48}&\to\C^{20}\\
  (C_{40},\ldots,C_{48})&\mapsto \text{coefficients of }1,q,\ldots,q^4\text{ of
  }\widetilde{g}_1,\ldots,\widetilde{g}_4.
\end{align*}
The space $\M_{40}\oplus\cdots\oplus\M_{48}$ has dimension $21$, thus the map
$\Phi$ has a non-trivial kernel. We choose $(C_{40},\ldots,C_{48})$ in the
kernel to form the operator $L_w$ (and then $\Kdown_w$). Then the functions
$L_wf_w,\widetilde{g}_1,\ldots,\widetilde{g}_4$ all vanish to order at least
$5$ at $i\infty$. If we divide these functions by $\Delta^5$, we still obtain
holomorphic functions. Thus $\Kdown_wf_w$ is a holomorphic quasimodular form of
weight $w-12$ and depth $4$. Furthermore, it vanishes to order (at least)
$\frac{5(w-12)}{12}$ at $i\infty$ and is thus extremal and solves
\eqref{eq:depth4-diff}\textsubscript{$w-12$}. The choice of the forms is given
in \eqref{eq:C40}--\eqref{eq:C48} in the Appendix.

With these two operators we obtain
\begin{align}
  \Kdown_{w+12}\Kup_wf_w&=\frac{2^{40}3^{23}}{5} (w+1) (w+2)^5 (w+3)^5 (w+4)^5
  (w+5) \notag\\ &\times (w+6)^4(w+7) (w+8)^5 (w+9)^5 (w+10)^5 (w+11)f_w.
  \label{eq:kdown-kup}
\end{align}
On the other hand by the construction of $\Kdown_w$ we have
\begin{equation*}
  \Kdown_wf_w=\frac{5^4}{2^4}
  w^4(5 w-12)^4 (5 w-24)^4(5 w-36)^4(5 w-48)^4  f_{w-12}.
\end{equation*}

Furthermore, the Frobenius ansatz gives
\begin{equation*}
  f_w=q^{\frac{5w}{12}}\Biggl(1+
    \frac{2 w \left(211 w^4+4440 w^3+12960 w^2-20736\right)}{(5 w+12)^4}q+\cdots
  \Biggr)
\end{equation*}
with explicit coefficients for $q^2,\cdots,q^6$.
From this we obtain
\begin{equation}\label{eq:depth4-vanishing}
  \begin{split}
    \bullet\,&\scriptstyle\partial_{w-4}f_w=\frac{w+1}3q^{\frac{5w}{12}}+\cdots\\[2mm]
    \bullet\,&\scriptstyle\frac{(w+1)(2w+1)}{18\vphantom{X^{X}}}E_4f_w-\partial_{w-4}^2f_w=
    \frac{2^63^5 (w+1) (w+2)^5}{(5 w+12)^4\vphantom{X^{X}}}q^{\frac{5w}{12}+1}+\cdots\\[2mm]
    \bullet\,&\scriptstyle\left(17 w^2+78 w+90\right)\partial_{w-4}^3f_w-
    \frac1{144}\left(191 w^4+1008 w^3+1504 w^2+192
      w-576\right)E_4\partial_{w-4}f_w\\
    &\scriptstyle-\frac1{432}(w+1) \left(81 w^4+376 w^3+560 w^2+528
      w+576\right)E_6f_w\\
    &\scriptstyle= \frac{2^{18}3^7 (w+1) (w+2)^5 (w+3)^5
      (w+4) (w+5)}{(5 w+12)^4 (5
      w+24)^4\vphantom{X^{X}}}q^{\frac{5w}{12}+2}+\cdots  \\[2mm]
    \bullet\,&\scriptstyle-\left(1313 w^6+28678 w^5+255122 w^4+1183008 w^3+3016512
      w^2+ 4012416 w+2177280\right)\partial_{w-4}^4f_w\\
    &\scriptstyle
    +\frac1{144}\bigl(13423 w^8+295800 w^7+2645368 w^6+12166080 w^5+29311504
      w^4+29020416
      w^3-15653376 w^2\\
      &\quad\scriptstyle-56692224 w-33094656\bigr)E_4\partial_{w-4}^2f_w\\
    &\scriptstyle+\frac1{432}\bigl(6561 w^9+136994 w^8+1139536
    w^7+4759344 w^6+10294016 w^5+11541472 w^4+14671104 w^3
    \\&\quad\scriptstyle+41398272
      w^2+63016704
      w+31974912\bigr)E_6\partial_{w-4}f_w\\
    &\scriptstyle+\frac1{2592}(w+1) \bigl(2048 w^9+38685 w^8+287792
    w^7+1130616 w^6+3110288 w^5+8497968 w^4\\
    &\quad\scriptstyle+18484992 w^3+14141952w^2-20570112 w-30855168\bigr)
    E_4^2f_w\\
    &\scriptstyle=\frac{2^{24}3^{13}(w+1) (w+2)^5 (w+3)^5 (w+4)^5 (w+5)
      (w+6)^4 (w+7)}{(5 w+12)^4 (5 w+24)^4(5w+36)^4\vphantom{X^{X}}}
    q^{\frac{5w}{12}+3}+\cdots \\[2mm]
    \bullet\,&\scriptstyle\left(293 w^4+4332 w^3+22968 w^2+51192 w+40824\right)E_4
    \partial_{w-4}^3f_w\\
    &\scriptstyle-\frac43\left(w^5+15 w^4+90 w^3+270 w^2+405 w+243\right)
    E_6\partial_{w-4}^2f_w\\
    &\scriptstyle-\frac1{144}\left(3311 w^6+51234 w^5+291550 w^4+731040
      w^3+717696 w^2-2592 w-256608\right)E_4^2\partial_{w-4}f_w\\
    &\scriptstyle-\frac1{432}(w+1) \left(1313 w^6+19430 w^5+104354
      w^4+251616 w^3+310464 w^2+300672 w+248832\right)E_4E_6f_w\\
    &\scriptstyle=\frac{2^{28}3^{14} (w+1) (w+2)^5 (w+3)^5 (w+4)
      (w+5) (w+6)^4 (w+7) (w+8)}{(5 w+12)^4 (5 w+24)^4 (5
      w+36)^4\vphantom{X^{X}}}q^{\frac{5w}{12}+3}+\cdots.
  \end{split}\raisetag{6.5cm}
\end{equation}
These are extremal quasimodular forms of respective weights
$w+2,\ldots,\allowbreak w+10$,
whose leading coefficient can be read off.
\begin{proposition}\label{prop:depth4}
  Let $(f_w)_{w\in2\mathbb{N}}$ ($w\geq12$) denote the sequence of normalized
  extremal quasimodular forms of depth $4$. Then for $w\equiv0\pmod{12}$
  \begin{equation}\label{eq:depth4-init}\textstyle
    f_{12}=\frac{13025 E_4^3-12796 E_6^2+
      3852 E_2E_4E_6-2706 E_2^2 E_4^2+27500 E_2^3 E_6-
      28875 E_2^4 E_4}{7449432883200\vphantom{X^{X}}}
  \end{equation}
  and
  \begin{equation}\label{eq:depth4-recurr}
    \scriptstyle
     f_{w+12}=
    \frac{5^5 (w+12)^4 (5 w+12)^4 (5 w+24)^4 (5 w+36)^4 (5 w+48)^4}
    {2^{44}3^{23} (w+1) (w+2)^5 (w+3)^5 (w+4)^5 (w+5) (w+6)^4 (w+7)
      (w+8)^5 (w+9)^5 (w+10)^5 (w+11)\vphantom{X^{X}}}\Kup_wf_w.
  \end{equation}
  Furthermore, the functions
  \begin{equation}
    \label{eq:depth4-2up}
    \begin{split}
      f_{w+2}&=\frac3{w+1}\partial_{w-4}f_w\\
      f_{w+4}&=\frac{(5 w+12)^4}{2^63^5 (w+1) (w+2)^5}
      \left(\frac{(w+1)(2w+1)}{18}E_4f_w-\partial_{w-4}^2f_w\right)\\
      f_{w+6}&=\frac{(5 w+12)^4 (5
        w+24)^4}{2^{18}3^7 (w+1) (w+2)^5 (w+3)^5 (w+4) (w+5)}\\
      &\times\left(\left(17 w^2+78 w+90\right)\partial_{w-4}^3f_w-\cdots\right)\\
           f_{w+8}&=
      \frac{(5 w+12)^4 (5 w+24)^4(5w+36)^4}
      {2^{24}3^{13}(w+1) (w+2)^5 (w+3)^5 (w+4)^5 (w+5)
        (w+6)^4 (w+7)}\\
      &\times\left(-\left(1313 w^6+\cdots+2177280\right)\partial_{w-4}^4f_w+\cdots\right)\\
         f_{w+10}&
      =\frac{(5 w+12)^4 (5 w+24)^4 (5 w+36)^4}
      {2^{28}3^{14} (w+1)
        (w+2)^5 (w+3)^5 (w+4) (w+5) (w+6)^4 (w+7) (w+8)}\\
      &\times\left(\left(293 w^4+4332 w^3+22968 w^2+51192 w+40824\right)E_4
 \partial_{w-4}^3f_w-\cdots\right)
    \end{split}\raisetag{4cm}
  \end{equation}
  are normalized extremal quasimodular forms of weights $w+2,\ldots,w+10$ (the
  omitted terms indicated by ``$\ldots$'' are given in
  \eqref{eq:depth4-vanishing}).
\end{proposition}
As an immediate consequence of
Propositions~\ref{prop-depth1}--~\ref{prop:depth4}, we obtain the following
fact conjectured in \cite{Kaneko_Koike2006:extremal_quasimodular_forms}.
\begin{theorem}\label{thm:kaneko-koike}
  The denominators of the coefficients of the normalized extremal quasimodular
  forms of weight $w$ and depth $\leq4$, are divisible only by primes $<w$.
\end{theorem}
\section*{Acknowledgment}
  The author is grateful to an anonymous referee for her/his very helpful
  comments on an earlier version of this paper. The referee suggested to
  formulate Proposition~\ref{prop:solution} as a separate statement and
  provided the reference \cite{Pellarin_Nebe2019:extremal_quasi_modular}, which
  contains a closely related study.

\appendix
\section{Coefficients of the operators for depth $4$}
\label{sec:coeff-oper-depth4}
In this appendix we collect the more elaborate formulas used especially in
Section~\ref{sec:depth-4}. The polynomials $p_0,\ldots,p_5$ in equations
\eqref{eq:depth4-p0}--\eqref{eq:depth4-p5} and the modular forms
$C_{40},\ldots,C_{48}$ in equations \eqref{eq:C40}--\eqref{eq:C48} were
computed with the help of \texttt{Mathematica}.
\begin{landscape}\nopagebreak
\begin{align}
    p_0(w)&=\scriptstyle 53567 w^{14}+4499628 w^{13}+173318340
    w^{12}+4055616864 w^{11}+
    64374205218 w^{10}+732790207224 w^9+6165100658404 w^8+38914973459904 w^7
    \notag\\
    &\scriptstyle+
    185044363180416 w^6
    +659055640624128 w^5+1729058937394176 w^4+
    3237068849283072 w^3
    +4084118362128384 w^2+3105388005949440 w\label{eq:depth4-p0}\\
    &\scriptstyle+1072718335180800\notag\\
\notag\\[-4mm]
    p_1(w)&=\scriptstyle21257 w^{11}+1465884
      w^{10}+45186990 w^9+821051740 w^8+9759703548 w^7
      +
      79588527156 w^6
      +453687847200 w^5\notag\\
      &\scriptstyle+1804779218520 w^4+4900200364800 w^3+
      8628400143360 w^2
      +8845395333120 w+3990767616000\label{eq:depth4-p1}\\
\notag\\[-4mm]
p_2(w)&=\scriptstyle2662740 w^{16}+224120550
    w^{15}+8648003840 w^{14}+202621853220 w^{13}+
  3217542322665 w^{12}
  +36586266504480 w^{11}+306658234963680 w^{10}\notag\\
  &\scriptstyle+
  1919356528986240 w^9
  +8970889439482816 w^8+30866477857195008 w^7+
  75319919247624192 w^6
  +118664936756305920 w^5\label{eq:depth4-p2}\\
    &\scriptstyle+83296021547483136 w^4
  -82769401579438080 w^3
  -258790551639293952 w^2-245119018746249216 w  -86822757140004864\notag\\
  \notag\\[-4mm]
p_3(w)&=\scriptstyle4272785 w^{17}+351970350
    w^{16}+13234823080 w^{15}+300533087760 w^{14}+4592608729932 w^{13}
    +49787752253076 w^{12}+392868254956864 w^{11}\notag\\
    &\scriptstyle  +2274866661846720 w^{10}+9597118952486912 w^9
    +28789901067644544 w^8  +58741997991303168 w^7
  +79017091035181056 w^6\label{eq:depth4-p3}\\
    &\scriptstyle
  +100071999240486912w^5+278562611915587584 w^4+779359222970449920 w^3
  +1260737947219525632 w^2+1054463073573666816 w\notag\\
    &\scriptstyle+355736061701259264\notag\\
 \notag\\[-4mm]
p_4(w)&=\scriptstyle517135 w^{17}+40772970 w^{16}
  +1455719580 w^{15}+31076826800 w^{14}+441034824168 w^{13}
  +4375275488634 w^{12}+31084796008256 w^{11}\notag\\
&\scriptstyle  +160090786631040 w^{10}+608772267089664 w^9
  +1834128793979392 w^8 +5229385586024448 w^7+15775977503047680 w^6
+40287913631023104  w^5\label{eq:depth4-p4}\\
&\scriptstyle+57115900062203904 w^4-19258645489385472 w^3
  -224285038806564864 w^2  -343616934723452928 w-182090547421249536\notag\\
\notag\\[-4mm]
p_5(w)&=\scriptstyle531441 w^{13}+36690686
    w^{12}+1133566168 w^{11}+20680195920 w^{10}+247548700336 w^9
    +2043291298652 w^8+11897624359104 w^7+49185666453888 w^6\notag\\
   &\scriptstyle +143692776009216 w^5
   +293687697411072 w^4+418695721574400 w^3+426532499288064 w^2
+316421756411904 w+135523565862912.\label{eq:depth4-p5}
\end{align}

\begin{equation}
  \label{eq:C40}
  \begin{split}
    C_{40}=&\scriptstyle212336640 (w-9) (w-3) \Bigl(26359
      w^{14}-2214156 w^{13}+85087955 w^{12}-1981432728 w^{11}+31214109018
      w^{10}
      -351608948568 w^9+2918019038293 w^8\\
      &\scriptstyle-18107342458608 w^7+84336226011558 w^6-293055231675096
      w^5+746938048608792 w^4-1352381369540544 w^3+1642191216192000 w^2\\
      &\scriptstyle-1195772294131200
      w+393589456128000\Bigr)E_4 E_6^6\\
+&\scriptstyle1658880 \Bigl(51900019 w^{16}-4982401824 w^{15}+221226817445
   w^{14}-6027400068900 w^{13}+112718884720404 w^{12}-1533303103010400
   w^{11}\\
   &\scriptstyle+15683910429124776 w^{10}
   -122977017767207520 w^9+746511198878954304
   w^8-3517545090663659520 w^7+12814220784607361280 w^6\\
   &\scriptstyle-35687146233130066944
   w^5
   +74437303893443933184 w^4-112350543174642769920
   w^3+115665398295339663360 w^2\\
   &\scriptstyle-72541839922746163200
   w+20875509918437474304\Bigr)E_4^4 E_6^4\\
+&\scriptstyle 1244160 \Bigl(79223933 w^{16}-7605497568
      w^{15}+337672506115 w^{14}-9198636642300 w^{13}+171985015024288
      w^{12}-2338768976800320 w^{11}\\
      &\scriptstyle+23913646971494624 w^{10}
  -187421212329029760 w^9+1137138517815826176
  w^8-5355376637586038784 w^7+19499550045226841088 w^6\\
  &\scriptstyle-54282475542685077504
  w^5+113193191971235303424 w^4-170842339568900800512 w^3
  +175944268854662529024 w^2\\
  &\scriptstyle-110441402045312532480 w
  +31829998802618548224\Bigr) E_4^7E_6^2\\
+&\scriptstyle1296 \Bigl(5782232065 w^{16}-555094278240
   w^{15}+24644070643200 w^{14}-671264352864000 w^{13}+12548411370416640
   w^{12}-170603540496691200 w^{11}\\
   &\scriptstyle+1743913567470202880
   w^{10}-13663290089279078400 w^9+82868952111275704320
   w^8-390124920526971863040 w^7\\
   &\scriptstyle+1419971443258483998720 w^6
   -3951666980565374730240 w^5+8238643867091026575360
   w^4-12434466652068469800960 w^3\\
   &\scriptstyle+12809293284994355036160 w^2
   -8045822842460386099200 w+2321594527761517510656\Bigr)
   E_4^{10}
  \end{split}
\end{equation}
\begin{equation}
  \label{eq:C42}
  \begin{split}
    C_{42}=
&\scriptstyle8847360 (w-10) (w-9) (w-7) (w-4) (w-3) (w-1) \Bigl(21257
   w^{11}-1340040 w^{10}+37636350 w^9-621010700 w^8+6681486588 w^7-49160979324
   w^6\\
   &\scriptstyle+252123174720 w^5-900257357160 w^4+2191152850560 w^3-3459266009856 w^2+3186365921280 w-1296999475200\Bigr)E_6^7\\
   +&\scriptstyle221184 (w-1) \Bigl(59652285 w^{16}-5727423730
   w^{15}+254372232445 w^{14}-6933224340880 w^{13}+129733172064000
   w^{12}-1766135709167940 w^{11}\\
   &\scriptstyle+18084440432438220
   w^{10}-141993070201020780 w^9+863455748585349240 w^8-4077640101288170880
   w^7\\
   &\scriptstyle+14896222070817504960 w^6
   -41630583090120877440 w^5+87211051870194381312
   w^4-132333305508808224768 w^3+137125544222548426752 w^2\\
   &\scriptstyle-86677899606946971648 w+25177794683564851200\Bigr)
   E_4^3E_6^5\\
+&\scriptstyle34560 (w-1) \Bigl(1163484751 w^{16}-111693353753 w^{15}+4959075896102
   w^{14}-135096500329618 w^{13}+2526019283240820 w^{12}\\
   &\scriptstyle-34353212893709184
   w^{11}
   +351290143724951040 w^{10}-2753475334119164160
   w^9+16707612532291757568 w^8-78689634714927968256 w^7\\
   &\scriptstyle+286520082131305949184
   w^6-797543727623713492992 w^5+1662730176637474947072 w^4-2508547795667083984896
   w^3\\
   &\scriptstyle+2581778110323244400640 w^2
   -1619027909322136879104 w+465981816686361182208\Bigr)
   E_4^6 E_6^3\\
+&\scriptstyle
    432 (w-1) \Bigl(28531797265 w^{16}-2738717899920 w^{15}+121567987037280
    w^{14}-3310530840743520 w^{13}+61865709853233600 w^{12}\\
    &\scriptstyle-840720468904181760 w^{11}+8588483182402488320
    w^{10}-67231802818857031680 w^9+407292115696721326080
    w^8-1914415680515583836160 w^7\\
    &\scriptstyle+6953430231459668459520
   w^6-19296803817086754816000 w^5+40082761664735503712256
   w^4-60204747917480347828224 w^3\\
   &\scriptstyle+61632482296473191448576 w^2-38403785466601122299904
   w+10969770793919657803776\Bigr)E_4^9 E_6
  \end{split}
\end{equation}
\begin{equation}
  \label{eq:C44}
  \begin{split}
    C_{44}=
   &\scriptstyle73728 \Bigl(12532755 w^{18}-1228914345 w^{17}+55932044285
   w^{16}-1568787184755 w^{15}+30362687859910 w^{14}-430260670247520 w^{13}\\
   &\scriptstyle+4622655124809540 w^{12}-38470537186614840
   w^{11}+251211075702837014 w^{10}-1295864332043809260 w^9+5290237710120191760
   w^8\\
   &\scriptstyle-17044689685346563680 w^7+43013914131829395456 w^6-83886321769065465408
   w^5+123718742229427747200 w^4-133286662841842990080 w^3\\
   &\scriptstyle+98925565236150067200 w^2-45228008659406118912
   w+9609196389953863680\Bigr) E_4^2E_6^6\\
   +&\scriptstyle23040 \Bigl(208766802
   w^{18}-20738706663 w^{17}+957610061114 w^{16}-27294230630727
   w^{15}+537821633898004 w^{14}-7775876558499438 w^{13}\\
   &\scriptstyle+85445305359966780 w^{12}-729304045355384352
   w^{11}+4899650334544430160 w^{10}-26095137555730257504
   w^9+110418717465901382976 w^8\\
   &\scriptstyle-370315908903799143936
   w^7+977199822897307158528 w^6-2002217065285338178560
   w^5+3117189914871002634240 w^4\\
   &\scriptstyle-3561207870202807173120 w^3+2814329391503205335040 w^2
   -1374584400002855337984 w+312699758864962682880\Bigr)E_4^5E_6^4\\
   + &\scriptstyle432 \Bigl(5640942435 w^{18}-574145660265
    w^{17}+27210408942670 w^{16}-797473001419560 w^{15}+16188499652022320
    w^{14}-241595348932913040 w^{13}\\
    &\scriptstyle+2745725576819892320 w^{12}-24286261064854871040
    w^{11}+169405063961957458432 w^{10}-938429115674081925120 w^9\\
    &\scriptstyle+4136645160744791040000
   w^8-14470987007027825786880 w^7+39867414308025880412160
   w^6-85316923540371471630336 w^5\\
   &\scriptstyle+138712136173982221271040 w^4-165366920569631924551680
   w^3+136184996014666316513280 w^2\\
   &\scriptstyle-69179329337614286192640 w+16327133797959811989504\Bigr)
   E_4^8E_6^2\\
+&\scriptstyle18 \Bigl(3825382395 w^{18}-417693655380 w^{17}+21201009227390 w^{16}-664296074843520 w^{15}+14390341237827840
   w^{14}-228724925396964480 w^{13}\\
   &\scriptstyle+2762631554345230080 w^{12}-25910717672216739840
   w^{11}+191177029926129983488 w^{10}-1117275590093461463040 w^9\\
   &\scriptstyle+5181258594279899381760 w^8-19011361988249752043520
   w^7+54763132291635013484544 w^6-122130979006167868243968 w^5\\
   &\scriptstyle+206227290808252579184640
   w^4-254460321591668401766400 w^3+216144824557113157091328 w^2\\
   &\scriptstyle-112869386433737376399360 w+27296067396795679899648\Bigr)
   E_4^{11}
  \end{split}
\end{equation}
\begin{equation}
  \label{eq:C46}
  \begin{split}
    C_{46}=
&\scriptstyle12288 (w-9) (w-3) \Bigl(21385 w^{17}-3000540 w^{16}+181158795 w^{15}-6431059740
   w^{14}+152509563792 w^{13}-2584774376916 w^{12}+32580209973804 w^{11}\\
   &\scriptstyle-312952105940940 w^{10}+2324178632140120 w^9-13441417608474240
   w^8+60599123972703360 w^7-211935600366998400 w^6\\
   &\scriptstyle+568296967045625280 w^5-1144426627632504960 w^4+1672115398874698752 w^3-1670883597388947456
   w^2\\
   &\scriptstyle+1019823559435689984 w-286322948470702080\Bigr) E_4E_6^7\\
&+\scriptstyle
    1536 (w-9) (w-3) \Bigl(43564885 w^{17}-4199424540 w^{16}+188153370170
    w^{15}-5202745378740 w^{14}+99416655438692 w^{13}-1392798686959116 w^{12}\\
    &\scriptstyle+14810213044144264 w^{11}-122061062138129400
    w^{10}+789187186132282416 w^9-4024771696392895584 w^8+16186815100766545920
    w^7\\
    &\scriptstyle-51035279993385265152
   w^6+124573554816237061632 w^5-230424866987544400896
   w^4+311845100542176927744 w^3\\
   &\scriptstyle-290906510930217861120 w^2+166976811547859681280 w
   -44392037471484641280\Bigr) E_4^4 E_6^5\\
+&\scriptstyle48 (w-3) \Bigl(5935658055 w^{18}-604130680965
   w^{17}+28644817180790 w^{16}-840407590446360 w^{15}+17090822555729336
   w^{14}\\
   &\scriptstyle-255749220384903192 w^{13}
   +2917535433452941424
   w^{12}-25935973597843047168 w^{11}+182093301447056161536
   w^{10}\\
   &\scriptstyle-1017033389895633205248 w^9
   +4528828858768286662656 w^8-16038621164741078679552
   w^7+44835034780192894009344 w^6\\
   &\scriptstyle-97590317268828859711488 w^5
   161770255432330054828032 w^4-197074103775825302913024 w^3+166176836346425804587008
   w^2\\
   &\scriptstyle-86567632937344309395456 w+20972979080173738524672\Bigr)
   E_4^7E_6^3\\
+&\scriptstyle3 (w-3) \Bigl(35224547665 w^{18}-3533943045270
   w^{17}+164951941875370 w^{16}-4757220455686080 w^{15}+94946288308939008
   w^{14}\\
   &\scriptstyle-1391879999449512576 w^{13}+15524090523960134912 w^{12}-134623421025963823104
   w^{11}+919695494473819799552 w^{10}\\
   &\scriptstyle-4984015509758024933376 w^9+21464444623798611689472
   w^8-73246940966966519660544 w^7+196469687001100291670016 w^6\\
   &\scriptstyle-408349070730564361519104 w^5+642737421946789924700160 w^4-738663893223616618168320 w^3+583143789019164227665920
   w^2\\
   &\scriptstyle-281900069458357354758144 w+62720415667791014658048\Bigr)
   E_4^{10}E_6\\
  \end{split}
\end{equation}
\begin{equation}
  \label{eq:C48}
  \begin{split}
    C_{48}=
   &\scriptstyle4096 (w-10) (w-9) (w-7)
   (w-4)^5 (w-3) (w-1) \Bigl(16 w^{10}-22025 w^9+1225150 w^8-30082800
   w^7+422010000 w^6-3708422460 w^5+21084612660 w^4\\
   &\scriptstyle-77240291040
     w^3+175004264880 w^2-221507758080 w+119022943536\Bigr) E_6^8\\
+&\scriptstyle512 (w-9) (w-4)^4 (w-3) \Bigl(2457393 w^{14}-206994196
    w^{13}+7967808727 w^{12}-185579433574 w^{11}+2918385315234
    w^{10}-32734570590132 w^9\\
    &\scriptstyle+269644653708324 w^8-1653949412659908 w^7+7574407203930120
    w^6-25704354625374456 w^5+63427818462081552 w^4\\
    &\scriptstyle-109936627329417984 w^3+125935036372661760
    w^2-84849181728168960 w+25173889100697600\Bigr) E_4^3E_6^6\\
+&\scriptstyle16 (w-4)^4
   \Bigl(372768821 w^{16}-35518302589 w^{15}+1564351553185
   w^{14}-42250630382915 w^{13}+782728613878590 w^{12}-10539973948217796
   w^{11}\\
   &\scriptstyle+106642134401044044 w^{10}-826431392449118340
   w^9+4953943715462672760 w^8-23029565126880689280 w^7+82687822898011670016
   w^6\\
   &\scriptstyle-226727406051812308224
   w^5+465078193737971627520 w^4-689454957866260193280
   w^3+696182683156551475200 w^2\\
   &\scriptstyle-427583296582966296576
   w+120287920179051823104\Bigr)E_4^6E_6^4\\
+&\scriptstyle(w-4)^4 \Bigl(2314158061 w^{16}-225295075074
    w^{15}+10156109605065 w^{14}-281276718421650 w^{13}+5354355751604718
    w^{12}\\
    &\scriptstyle-74250742207449696 w^{11}+775561858607479616 w^{10}-6221336083117357440
    w^9+38716279294460726016 w^8-187454904571023464448 w^7\\
    &\scriptstyle+703499359272507740160 w^6-2024100636419312971776
    w^5+4375390173868546007040 w^4-6867228800076590481408 w^3\\
    &\scriptstyle+7378429565839006236672 w^2-4847796658391713579008
    w+1467072833894567903232\Bigr) E_4^9E_6^2\\
+&\scriptstyle\frac{3}{16} (w-4)^4 \Bigl(177147 w^{16}-4094579648 w^{15}+374950410640 w^{14}-15829704501920 w^{13}+408049598642272
   w^{12}-7176728984678912 w^{11}\\
   &\scriptstyle+91165949198565376 w^{10}-863556865787120640
   w^9+6207504451510781952 w^8-34121559189028110336 w^7+143351014974408622080
   w^6\\
   &\scriptstyle-456319260766953603072 w^5+1080402954748938289152
   w^4-1841108952145673060352 w^3+2131230194096535502848 w^2\\
   &\scriptstyle-1498210290755300229120
   w+482085103232027197440\Bigr) E_4^{12}\\
  \end{split}
 \end{equation}
\end{landscape}

\end{document}